\documentclass[11pt]{amsart}
\usepackage[top=1in, bottom=1in, left=1in, right=1in]{geometry}

\usepackage{mathtools,mathrsfs,amsthm,amsfonts,amssymb,graphicx,appendix,cancel,caption,subcaption}
\usepackage{xspace,enumerate,comment,bbm}
\usepackage[super]{nth}

\usepackage{tikz,pgfplots}

\usepackage{hyperref}
\usepackage[normalem]{ulem}
\usepackage[active]{srcltx}

\numberwithin{equation}{section}
\newtheorem{theorem}{\sc Theorem}[section]

\newtheorem{corollary}[theorem]{\sc Corollary}
\newtheorem{definition}[theorem]{\sc Definition}
\newtheorem{lemma}[theorem]{\sc Lemma}
\newtheorem{proposition}[theorem]{\sc Proposition}

\theoremstyle{remark}
\newtheorem{remark}[theorem]{\sc Remark}

\newcommand{\ol}[1]{\overline{#1}}
\newcommand{\ul}[1]{\underline{#1}}

\newcommand{\R}{\mathbb{R}}
\newcommand{\cal}{\mathcal}

\renewcommand{\epsilon}{\varepsilon}
\renewcommand{\emptyset}{\varnothing}
\renewcommand{\phi}{\varphi}
\renewcommand{\le}{\leqslant}
\renewcommand{\ge}{\geqslant}

\newcommand{\Con}{\operatorname{C}}
\newcommand{\UC}{\operatorname{UC}}
\newcommand{\Lip}{\operatorname{C^{0,1}}}
\newcommand{\Liploc}{\operatorname{C_{loc}^{0,1}}}

\makeatletter
\@namedef{subjclassname@2020}{\textup{2020} Mathematics Subject Classification}
\makeatother

\begin{document}
	
\title[Homogenization of viscous HJ equations]{Homogenization of nonconvex viscous Hamilton-Jacobi equations in stationary ergodic media in one dimension}
	
\author[E.\,Kosygina]{Elena Kosygina}
\address{Elena Kosygina\\ Department of Mathematics\\ Baruch College\\  One Bernard Baruch Way\\ Box B6-230, New York, NY 10010\\ USA}
\email{elena.kosygina@baruch.cuny.edu}
\urladdr{http://www.baruch.cuny.edu/math/elenak/}
\thanks{E.\,Kosygina was partially supported by the Simons Foundation (Award \#523625).}
	
\author[A.\,Yilmaz]{Atilla Yilmaz}
\address{Atilla Yilmaz\\ Department of Mathematics\\ Temple University\\ 1805 North Broad Street, Philadelphia, PA 19122, USA}
\email{atilla.yilmaz@temple.edu}
\urladdr{http://math.temple.edu/$\sim$atilla/}
\thanks{A.\,Yilmaz was partially supported by the Simons Foundation (Award \#949877).}
	
\date{March 23, 2024}
	

\begin{abstract} 

We establish homogenization for nondegenerate viscous Hamilton-Jacobi equations in one space dimension when the diffusion coefficient $a(x,\omega) > 0$ and the Hamiltonian $H(p,x,\omega)$ are general stationary ergodic processes in $x$. Our result is valid under mild regularity assumptions on $a$ and $H$ plus standard coercivity and growth assumptions (in $p$) on the latter. In particular, we impose neither any additional condition on the law of the media nor any shape restriction on the graph of $p\mapsto H(p,x,\omega)$. Our approach consists of two main steps: (i) constructing a suitable candidate $\ol H$ for the effective Hamiltonian; (ii) proving homogenization. In the first step, we work with the set $E$ of all points at which $\ol H$ is naturally determined by correctors with stationary derivatives. We prove that $E$ is a closed subset of $\R$ that is unbounded from above and below, and, if $E\neq \R$, then $\ol H$ can be extended continuously to $\R$ by setting it to be constant on each connected component of $E^c$. In the second step, we use a key bridging lemma, comparison arguments and several general results to verify that homogenization holds with this $\ol H$ as the effective Hamiltonian.

\end{abstract}
	
\maketitle

\section{Introduction}

We consider the homogenization problem for a family of viscous
Hamilton-Jacobi (HJ) equations (as $\epsilon\to 0$)
\begin{equation}\label{eq:hjt}
	\partial_t u^\epsilon = \epsilon a\left(\frac{x}{\epsilon},\omega\right) \partial^2_{xx} u ^\epsilon+H\left(\partial_x u^\epsilon,\frac{x}{\epsilon},\omega\right),\quad (t,x)\in (0,\infty)\times\R,
\end{equation}
with $u^\epsilon(0,x,\omega)=g(x),\ x\in\R$, where the diffusion
coefficient $a(x,\omega) > 0$ and the Hamiltonian $H(p,x,\omega)$ are
stationary ergodic processes with respect to the shifts in $x\in\R$.

Under a set of standard coercivity and growth assumptions in the
momentum variable $p$ on $H$ and regularity assumptions on $a$ and $H$
under the shifts in $x$ (see Section \ref{sub:stand} and
  Remark~\ref{rem:suff}), we show that equation \eqref{eq:hjt}
homogenizes. This means that there is a continuous coercive nonrandom
function $\ol{H}:\R\to\R$, the effective Hamiltonian, such that, with
probability 1, for every uniformly continuous function $g$ on $\R$,
the solutions $u^\epsilon$ of \eqref{eq:hjt} satisfying
$u^\epsilon(0,\,\cdot\,,\omega) = g$ converge locally uniformly on
$[0,\infty)\times \R$ as $\epsilon\to 0$ to the unique solution
$\ol{u}$ of the equation
\[ \partial_t \ol{u} = \ol{H}(D\ol{u}), \quad (t,x)\in (0,\infty)\times\R. \]
Our work has three main features:
\begin{itemize}
\item [(i)] we do not impose any shape restriction, such as convexity or quasiconvexity, on the graph of $p\mapsto H(p,x,\omega)$;
\item [(ii)] we do not make any additional assumptions on the stationary ergodic media and treat all cases, be it periodic or stochastic, in a unified way;
\item [(iii)] our arguments are general and they do not use any decomposition, induction, or gluing procedures based on the specifics of $a$ and $H$.
\end{itemize}

We recall that convexity of $H$ in the momentum variable plays no role
in the periodic homogenization of HJ equations.
Using the compactness of the torus, general homogenization results in all dimensions have been established in
\cite{LPV, Evans92}.
The non-compact stationary ergodic setting offers many challenges and requires new methods. For the past 25 years, it has been the subject of extensive research.
Below, to stay focused, we only mention a fraction of the pertaining literature.

When the Hamiltonian is convex in $p$, stochastic homogenization
of inviscid and viscous HJ equations in all spatial
dimensions ($x\in\R^d$, $d\ge 1$) is well-known, \cite{Sou99, RT,
  LS_viscous, KRV, LS_revisited, AS12, AT14}. The effective
Hamiltonian is convex.

In the inviscid case, stochastic homogenization results have been
extended to Hamiltonians which are quasiconvex in $p$ (i.e., all
sublevel sets $\{p\in\R^d:\,H(p,x,\omega)\le \lambda\}$ are convex),
\cite{DS09} ($d=1$), \cite{AS} ($d\ge 1$). The effective Hamiltonian
in this case is quasiconvex. In the viscous case, as we recently
demonstrated in \cite{KY23+}, quasiconvexity of $H$ does not imply
quasiconvexity of $\ol{H}$. Except for $d=1$ (\cite{Y21b,D24+a,D24+b},
and the current paper), the question whether viscous HJ equations with
quasiconvex Hamiltonians homogenize remains open.

Over the last decade, homogenization was proven for many different
classes of stochastic HJ equations with non-quasiconvex Hamiltonians,
\cite{ATY_1d, Gao16, DK17, YZ19, KYZ20, DK22, DKY23+} ($d=1$) and
\cite{ATY_nonconvex, CaSo17, FS, AC18, QTY, Gao19} ($d\ge
1$). Nevertheless, counterexamples show that, in both inviscid and
viscous cases, homogenization for $d\ge 2$ can fail if the otherwise
standard coercive Hamiltonian $H(\cdot,x,\omega)$ has a strict saddle
point, \cite{Zil, FS, FFZ}.

These positive and negative results reinforced our interest in completing the
qualitative homogenization picture in the one-dimensional viscous
case. We recall that, when $a\equiv 0$, the starting point of the analysis
in \cite{ATY_1d, Gao16} was the already mentioned fact that
\eqref{eq:hjt} with quasiconvex $H$ homogenizes and produces a
quasiconvex $\ol{H}$. The general nonconvex $H$ was treated by
approximation, decomposition according to the relative magnitudes of the fluctuations of $H$ in $p$ and $x$, and careful gluing. Carrying out a
similar procedure in the viscous case presented serious obstacles. To
start, one needed a result for viscous HJ equations with quasiconvex Hamiltonians. This was partially
achieved in \cite{Y21b}, which showed homogenization for
$a(x,\omega)>0$ and $H(p,x,\omega)=G(p)+V(x,\omega)$ with $G$
quasiconvex under standard assumptions plus an additional {\em scaled
  hill condition} on the diffusion-potential pair $(a,V)$.
We note that the {\em hill and valley condition} was first introduced
in \cite{YZ19, KYZ20}. Its weaker {\em scaled} version (see (S) and
Remark 2.3 in \cite{DKY23+} and Appendix B of \cite{DK22}) combined
with the result of \cite{CaSo17} on the existence of correctors allowed to
significantly expand the class of homogenizable equations
\eqref{eq:hjt}, \cite{DK22}. Recently, starting from the result of
\cite{Y21b}, in a joint work with A.\,Davini we were able to further
adapt the decomposition method and prove homogenization for
$a(x,\omega)>0$ and $H(p,x,\omega)=G(p)+V(x,\omega)$ under the scaled
hill and valley condition on $(a,V)$ for general $G$ with superlinear
growth, \cite{DKY23+}.

The wish to remove the scaled hill and valley condition prompted us to return to
the special case of periodic environments which by default do not satisfy this condition. The effort resulted in
the already mentioned paper \cite{KY23+}, which demonstrated, in
particular, that homogenization of equations \eqref{eq:hjt} with
$a\equiv 1$ and Hamiltonians not satisfying the scaled hill and
valley condition could produce effective Hamiltonians that are
qualitatively different than those appearing in the inviscid case
or under the scaled hill and valley condition. This discovery brought us
back to the ``drawing board'' and fueled an attempt to construct
$\ol{H}$ solely from the information collected from correctors, prove
homogenization for linear initial data, and use the results of
\cite{DK17} to establish homogenization for general $g$.

We succeeded in doing that for \eqref{eq:hjt} with $a>0$. As we mentioned
at the beginning, our approach relies only on standard assumptions on
$a$ and $H$ and does not use any information about the local extrema of
$H$ in $p$ or the law of the random medium (apart from stationarity and
ergodicity). The trade-off is that it yields only basic
properties of the effective Hamiltonian. Namely, $\ol{H}$ is coercive,
locally Lipschitz, and satisfies the same upper and lower bounds in
$p$ as the original $H$. This is similar to the case of
corrector-based periodic homogenization (which is naturally covered
under our assumptions). While our arguments crucially depend on
the nondegeneracy assumption $a>0$, we believe that we finally ``got
it right'' as our approach is both general and relatively simple.  

In conclusion, we would like to remark that the ``decomposition and
gluing method'' developed in \cite{ATY_1d,Gao16,QTY,Y21a} for the
inviscid case and extended to the viscous case under various sets of
assumptions in the cited above papers provides much more details about
$\ol{H}$. For example, $\ol{H}$ inherits
quasiconvexity from $H$ when $a>0$, $H(p,x,\omega)=G(p)+V(x,\omega)$,
and $(a,V)$ satisfies the scaled hill condition, \cite{Y21b}, or
when $\min_{\R}a(\,\cdot\,,\omega)=0$ and $H$ is general superlinear in
$p$, \cite{DS09, D24+b}.

\smallskip

{\bf Organization of the paper.} In Sections~\ref{sub:stand} and
\ref{sub:main} we state our assumptions and the main
result. Section~\ref{sub:ps} gives the outline of the proof. The
ingredients of the proof include existence of solutions of the ODE for
the derivatives of correctors (Section~\ref{sec:ode}) and construction
and properties of $\ol{H}$ including an important bridging lemma
(Section~\ref{sec:effective}). The proof of homogenization is given in
Section~\ref{sec:homogenization}.

\section{Main result}\label{sec:main}

\subsection{Standing assumptions}\label{sub:stand}

Throughout the paper, for any domain of the form $X = I_1\times I_2$ or $X = I_2$, where $I_1\subset[0,\infty)$ and $I_2\subset\R$ are open intervals,
$\Con(X)$, $\UC(X)$, $\Lip(X)$ and $\Liploc(X)$ stand for the space of continuous, uniformly continuous, Lipschitz continuous and locally Lipschitz continuous real-valued functions on $X$, respectively.
Similarly, $\Con^k(X)$, $k\in\{1,2\}$, stand for the space of real-valued functions on $X$ with continuous derivatives of order $k$. These definitions extend to the closure of $X$ as usual.

Let $(\Omega,\mathcal{B}(\Omega))$ be a Polish space endowed with its Borel
$\sigma$-algebra, $\mathbb{P}$ a complete probability measure on
$(\Omega,\mathcal{F})$, where $\mathcal{F}$ is the completion of $\mathcal{B}(\Omega)$ with respect to $\mathbb{P}$, and $\tau_x:\Omega\to\Omega$, $x\in\R$, a
group of measure-preserving transformations such that
$(x,\omega)\mapsto \tau_x\omega$ is $(\mathcal{B}(\R)\otimes\mathcal{F})/\mathcal{F}$-measurable. Assume that
$\mathbb{P}$ is ergodic under this group of transformations, i.e., if
$A\in {\cal F}$ satisfies $\cap_{x\in\R}\tau_xA=A$, then
$\mathbb{P}(A) \in\{0,1\}$.  We write $\mathbb{E}[\,\cdot\,]$ to denote
expectation with respect to $\mathbb{P}$.

We consider a nondegenerate viscous Hamilton-Jacobi (HJ) equation
\begin{equation}\label{eq:parabol}
	\partial_t u^\epsilon = \epsilon a\left(\frac{x}{\epsilon},\omega\right) \partial^2_{xx} u ^\epsilon+H\left(\partial_x u^\epsilon,\frac{x}{\epsilon},\omega\right),  \quad (t,x)\in (0,\infty)\times\R,
\end{equation}
with $\epsilon > 0$ and $\omega\in\Omega$, where the measurable functions
\[  a:\R\times\Omega\to(0,1]\quad\text{and}\quad H:\R\times\R\times\Omega\to\R \]
satisfy the following conditions:
\begin{align}
	&\text{$a$ is stationary, i.e.,}\ a(x,\omega) = a(0,\tau_x\omega)\ \text{for all}\ (x,\omega)\in\R\times\Omega;\tag{$\mathrm{A1}$}\label{A1}\\
	&\text{there exists a $\kappa>0$ such that}\tag{$\mathrm{A2}$}\label{A2}\\
	&\left|\sqrt{a(x,\omega)} - \sqrt{a(0,\omega)}\right| \le \kappa|x|\ \text{for all}\ (x,\omega)\in\R\times\Omega;\nonumber\\
	&H(p,\,\cdot\,,\,\cdot\,)\ \text{is stationary, i.e.,}\ H(p,x,\omega) = H(p,0,\tau_x\omega)\ \text{for all}\ (p,x,\omega)\in\R\times\R\times\Omega;\tag{$\mathrm{H1}$}\label{H1}\\
	&\text{there exist even functions $G_L,G_U\in\Con(\R)$ that are nondecreasing on $[0,\infty)$, coercive, i.e.,}\tag{$\mathrm{H2}$}\label{H2}\\
	&\lim_{p\to\infty}G_L(p) = \lim_{p\to\infty}G_U(p) = \infty,\ \text{and}\ G_L(p) \le H(p,0,\omega) \le G_U(p)\ \text{for all}\ (p,\omega)\in\R\times\Omega;\nonumber\\
	&\text{for every $R>0$, there exists a $K_R>0$ such that}\tag{$\mathrm{H3}$}\label{H3}\\
	&|H(p,0,\omega)-H(q,0,\omega)|\le K_R |p-q|\ \text{for all $(p,q,\omega)\in[-R,R]\times[-R,R]\times\Omega$};\nonumber\\
	&\text{for every $R>0$, there exists a modulus of continuity $m_R:[0,\infty)\to[0,\infty)$ such that}\tag{$\mathrm{H4}$}\label{H4}\\
	&|H(p,x,\omega) - H(p,0,\omega)| \le m_R(|x|)\ \text{for all $(p,x,\omega)\in [-R,R]\times\R\times\Omega$.}\nonumber
\end{align}

\subsection{Statement of the main result}\label{sub:main}

When $\epsilon = 1$, we drop the superscript of $u^\epsilon$ in \eqref{eq:parabol} and write
\begin{equation}\label{eq:birhuzur}
	\partial_tu = a(x,\omega)\partial_{xx}^2u + H(\partial_xu,x,\omega),\quad (t,x)\in(0,\infty)\times\R.
\end{equation}

Our main result establishes the homogenization of \eqref{eq:parabol} with linear initial data. To this end, we make the following additional assumptions (see Remark \ref{rem:suff} for sufficient conditions):
\begin{align}
	&\begin{aligned}\label{eq:existunique}
		&\text{for every $\omega\in\Omega$ and $\theta\in\R$, \eqref{eq:birhuzur} has a unique viscosity solution}\\
		&\text{$u_\theta(\,\cdot\,,\,\cdot\,,\omega)\in\UC([0,\infty)\times\R)$ such that $u_\theta(0,x,\omega) = \theta x$ for all $x\in\R$;}
	\end{aligned}\\
	&\begin{aligned}\label{eq:duzlip}
		&\text{for every $\omega\in\Omega$ and $\theta\in\R$, there exists an $\ell_\theta = \ell_\theta(\omega) > 0$ such that}\\
		&|u_\theta(t,x,\omega) - u_\theta(t,y,\omega)| \le \ell_\theta|x-y|\ \ \text{for all}\ t\in[0,\infty)\ \text{and}\ x,y\in\R.
	\end{aligned}
\end{align}
Note that these assumptions carry over to \eqref{eq:parabol} with an arbitrary $\epsilon > 0$.
Indeed, the unique viscosity solution of the latter equation with the same initial condition is given by
\[ u_\theta^\epsilon(t,x,\omega) = \epsilon u_\theta\left(\frac{t}{\epsilon},\frac{x}{\epsilon},\omega\right). \]
See Section \ref{sub:viscosity} for some preliminaries regarding viscosity solutions of \eqref{eq:birhuzur}.

Here is the statement of our main result.

\begin{theorem}\label{thm:homlin}
	Assume that \eqref{A1}, \eqref{A2}, \eqref{H1}--\eqref{H4}, \eqref{eq:existunique} and \eqref{eq:duzlip} hold. Then, for $\mathbb{P}$-a.e.\ $\omega$, as $\epsilon\to0$, \eqref{eq:parabol} with linear initial data homogenizes to an inviscid HJ equation
	\begin{equation}\label{eq:efhuzur}
		\partial_t\ol u = \ol H(\partial_x\ol u),\quad(t,x)\in(0,\infty)\times\R,
	\end{equation}
	with some coercive $\ol H\in\Liploc(\R)$. More precisely, there exists an $\Omega_0\in\mathcal{F}$ with $\mathbb{P}(\Omega_0) = 1$ such that, for every $\omega\in\Omega_0$ and $\theta\in\R$, as $\epsilon\to0$, the unique viscosity solution $u_\theta^\epsilon(\,\cdot\,,\,\cdot\,,\omega)$ of \eqref{eq:parabol} with the initial condition $u_\theta^\epsilon(0,x,\omega) = \theta x$, $x\in\R$, converges locally uniformly on $[0,\infty)\times\R$ to $\ol u_\theta$ defined by
	\begin{equation}\label{eq:verya}
		\ol u_\theta(t,x) = t\ol H(\theta) + \theta x,
	\end{equation}
	which is the unique solution of \eqref{eq:efhuzur} with the same initial condition.
\end{theorem}

Finally, replacing \eqref{eq:existunique} with the stronger assumption that
\begin{equation}\label{eq:well}
	\begin{aligned}
		&\text{the Cauchy problem for \eqref{eq:birhuzur} is well-posed in $\UC([0,\infty)\times\R)$ for every $\omega\in\Omega$,}
	\end{aligned}
\end{equation}
which is defined in Section \ref{sub:viscosity}, we generalize Theorem \ref{thm:homlin} to uniformly continuous initial data by citing a result from \cite{DK17} which is based on the perturbed test function method (see \cite{Evans89}).

\begin{corollary}\label{cor:velinim}
	Assume that \eqref{A1}, \eqref{A2}, \eqref{H1}--\eqref{H4}, \eqref{eq:duzlip} and \eqref{eq:well} hold. Then, for $\mathbb{P}$-a.e.\ $\omega$, as $\epsilon\to0$, \eqref{eq:parabol} with uniformly continuous initial data homogenizes to the inviscid HJ equation \eqref{eq:efhuzur}. More precisely, there exists an $\Omega_0\in\mathcal{F}$ with $\mathbb{P}(\Omega_0) = 1$ such that, for every $\omega\in\Omega_0$ and $g\in\UC(\R)$, as $\epsilon\to0$, the unique viscosity solution $u_g^\epsilon(\,\cdot\,,\,\cdot\,,\omega)$ of \eqref{eq:parabol} with the initial condition $u_g^\epsilon(0,\,\cdot\,,\omega) = g$ converges locally uniformly on $[0,\infty)\times\R$ to the unique viscosity solution $\ol u_g$ of \eqref{eq:efhuzur} with the same initial condition.
\end{corollary}

\begin{remark}\label{rem:suff}
	On top of \eqref{A1}, \eqref{A2} and \eqref{H1}, if there exist constants $\alpha_0,\beta_0>0$ and $\gamma>1$ such that
	\begin{align}
		\alpha_0|p|^\gamma - 1/{\alpha_0} &\le H(p,0,\omega) \le \beta_0(|p|^\gamma + 1)\quad\text{for all}\ (p,\omega)\in\R\times\Omega,\tag{$\mathrm{H2'}$}\label{eq:H2'}\\
		|H(p,0,\omega) - H(q,0,\omega)| &\le \beta_0(|p| + |q| + 1)^{\gamma - 1}|p- q|\quad\text{for all}\ (p,q,\omega)\in\R\times\R\times\Omega,\quad\text{and}\tag{$\mathrm{H3'}$}\label{eq:H3'}\\
		|H(p,x,\omega) - H(p,0,\omega)| &\le \beta_0(|p|^\gamma + 1)|x|\quad\text{for all}\ (p,x,\omega)\in\R\times\R\times\Omega,\tag{$\mathrm{H4'}$}\label{eq:H4'}
	\end{align}
	then the additional assumptions \eqref{eq:existunique}, \eqref{eq:duzlip} and \eqref{eq:well} in Theorem \ref{thm:homlin} and Corollary \ref{cor:velinim} are valid (see \cite[Theorem 2.8]{DK17} and \cite[Theorem 3.2]{D19}). Note that \eqref{eq:H2'}, \eqref{eq:H3'} and \eqref{eq:H4'} are stronger versions of \eqref{H2}, \eqref{H3} and \eqref{H4}, respectively. These sufficient conditions (which we provide here for the sake of completeness) are not used anywhere in this paper.
\end{remark}

\subsection{Outline of the proof}\label{sub:ps}

In a nutshell, our approach consists of first constructing a suitable candidate for the effective Hamiltonian $\ol H(\theta)$, $\theta\in\R$, and then proving the convergence stated in Theorem \ref{thm:homlin}. Below we flesh out these two steps.

Given any $\theta\in\R$, assume that there exist
a $\lambda(\theta)\in\R$ and a stationary $f_\theta:\R\times\Omega\to\R$ such that $\mathbb{E}[f_\theta(0,\omega)] = \theta$ and 
\begin{equation}\label{eq:rey}
	a(x,\omega)f_\theta'(x,\omega) + H(f_\theta(x,\omega),x,\omega) = \lambda(\theta)
\end{equation}
for all $x\in\R$ and $\mathbb{P}$-a.e.\ $\omega$.
Let
\[ F_\theta(x,\omega) = \int_0^xf_\theta(y,\omega)dy. \]
Then, $\tilde u_\theta(t,x,\omega) = t\lambda(\theta) + F_\theta(x,\omega)$ is a solution of \eqref{eq:birhuzur}, and $F_\theta(x,\omega) = \theta x + o(x)$ as $|x|\to\infty$ by Birkhoff's ergodic theorem. (The function $x\mapsto F_\theta(x,\omega) - \theta x$ is known in the literature as a sublinear corrector, but we will not use this terminology in the rest of the paper as we will always work directly with $f_\theta$ or $F_\theta$.)

We use $\lambda(\theta)$ to construct our candidate for the effective Hamiltonian. To this aim, let $E$ be the set of all $\theta\in\R$ for which the assumption in the above paragraph is valid. Using our standing assumptions \eqref{A1}, \eqref{A2} and \eqref{H1}--\eqref{H4}, we show in Lemma \ref{lem:support} that $E$ is nonempty, $\inf E = -\infty$ and $\sup E = \infty$. Moreover, $E$ is a closed subset of $\R$  (see Lemma \ref{lem:closed}). If $E = \R$, then we simply set $\ol H(\theta) = \lambda(\theta)$ for all $\theta\in\R$. Otherwise, we can write $E^c$ as a disjoint union of at most countably many bounded open intervals:
\begin{equation*}
	E^c = \bigcup_{j\in J}(\theta_L^{(j)},\theta_R^{(j)}).
\end{equation*}
We prove in Lemma \ref{lem:esit} that $\lambda(\theta_L^{(j)}) = \lambda(\theta_R^{(j)}) =: \ol\lambda_j$ for all $j\in J$. This allows us to define $\ol H$ by setting
\[  \ol H(\theta) =
\begin{cases} \lambda(\theta)&\text{if}\ \theta\in E,\\
	\ol\lambda_j &\text{if}\ \theta\in(\theta_L^{(j)},\theta_R^{(j)})\ \text{for any}\ j\in J.
\end{cases}
\]
Lemma \ref{lem:lam} shows that $\ol{H}$ is locally Lipschitz continuous.

The proof of Lemma \ref{lem:esit} illustrates a key idea in the paper, so we provide a sketch of it here.
Suppose that $\lambda(\theta_L^{(j)}) \neq \lambda(\theta_R^{(j)})$ for some $j\in J$. For the sake of convenience, we write
\begin{equation}\label{eq:conv}
	\theta_1 = \theta_L^{(j)},\quad \theta_2=\theta_R^{(j)},\quad f_1=f_{\theta_1}\quad\text{and}\quad f_2=f_{\theta_2}.
\end{equation}
We show the existence of a stationary $f_0:\R\times\Omega\to\R$ such that
\[ a(x,\omega)f_0'(x,\omega) + H(f_0(x,\omega),x,\omega) = \frac12(\lambda(\theta_1) + \lambda(\theta_2)) \]
and $f_1(x,\omega) < f_0(x,\omega) < f_2(x,\omega)$ for all $x\in\R$ and $\mathbb{P}$-a.e.\ $\omega$.
Recalling equation \eqref{eq:rey} and the definition of $E$, we deduce that $\mathbb{E}[f_0(0,\omega)] \in E\cap(\theta_1,\theta_2) = \emptyset$, which is  a contradiction.

Our next task is to prove that, under the additional assumption \eqref{eq:existunique}, for every $\theta\in\R$,
\begin{equation}\label{eq:nebiye}
		\lim_{\epsilon\to0}u^\epsilon_\theta(1,0,\omega) = \lim_{t\to \infty}t^{-1}u_\theta(t,0,\omega)=\ol H(\theta)\ \ \text{for $\mathbb{P}$-a.e.\ $\omega$.}
\end{equation}
There are two cases:

\smallskip

(i) If $\theta\in E$, and, hence, $\ol H(\theta) = \lambda(\theta)$, then \eqref{eq:nebiye} follows from standard comparison arguments involving $u_\theta(t,x,\omega)$ and $\tilde u_\theta(t,x,\omega) = t\ol H(\theta) + F_\theta(x,\omega)$ (see Lemma \ref{lem:isbu}). 

\smallskip

(ii) If $\theta\not\in E$, then there is a $j\in J$ such that $\theta\in(\theta_L^{(j)},\theta_R^{(j)})$.
In Lemma \ref{lem:kibbe}, we show (using a proof by contradiction similar to the one for Lemma \ref{lem:esit} which is sketched above) that there exists an $\ol\Omega_j\in\mathcal{F}$ with $\mathbb{P}(\ol\Omega_j) = 1$ such that, for every $\omega\in\ol\Omega_j$ and $\delta\in(0,1)\cap\mathbb{Q}$, there exist constants
$\ul z_1 < \ol z_2$ and a function $f_{1,2}^{\ol\lambda_j + \delta}(\,\cdot\,,\omega)\in\Con^1([\ul z_1,\ol z_2])$ such that
\[ a(x,\omega)(f_{1,2}^{\ol\lambda_j + \delta})'(x,\omega) + H(f_{1,2}^{\ol\lambda_j + \delta}(x,\omega),x,\omega) = \ol\lambda_j + \delta\ \text{for all $x\in[\ul z_1,\ol z_2]$}, \]
$f_{1,2}^{\ol\lambda_j + \delta}(\ul z_1,\omega) = f_1(\ul z_1,\omega)$ and $f_{1,2}^{\ol\lambda_j + \delta}(\ol z_2,\omega) = f_2(\ol z_2,\omega)$ with the shorthand notation in \eqref{eq:conv}.
Then, in Lemma \ref{lem:LBUB}, we use $f_{1,2}^{\ol\lambda_j + \delta}(\,\cdot\,,\omega)$ to construct a ``bridge" between appropriately shifted versions of $F_{\theta_1}(\,\cdot\,,\omega)$ and $F_{\theta_2}(\,\cdot\,,\omega)$. More precisely, we define a function $F_{1,2}^{\ol\lambda_j + \delta}(\,\cdot\,,\omega)\in\Lip\cap\Con^1(\R)$ by setting
\[ F_{1,2}^{\ol\lambda_j + \delta}(x,\omega) = \begin{cases}
		\int_{\ul z_1}^xf_1(y,\omega)\,dy &\text{if}\ x \le \ul z_1,\\ 
		\int_{\ul z_1}^xf_{1,2}^{\ol\lambda_j + \delta}(y,\omega)dy &\text{if}\ \ul z_1 < x < \ol z_2,\\
		\int_{\ul z_1}^{\ol z_2}f_{1,2}^{\ol\lambda_j + \delta}(y,\omega)dy +\int_{\ol z_2}^xf_2(y,\omega)\,dy &\text{if}\ x \ge \ol z_2.
	\end{cases}
\]
It follows that \[ \tilde w_{j,\delta}(t,x,\omega) = t(\ol\lambda_j + \delta) + F_{1,2}^{\ol\lambda_j + \delta}(x,\omega) \] is a viscosity supersolution of \eqref{eq:birhuzur}.
Recalling that $\theta_1 < \theta < \theta_2$ and comparing $u_\theta$ and $\tilde w_{j,\delta}+C$ we deduce that \[\limsup_{t\to \infty}t^{-1}u_\theta(t,0,\omega) \le \limsup_{t\to \infty}t^{-1}(\tilde w_{j,\delta}(t,0,\omega)+C)=\ol\lambda_j + \delta.\] Since $\delta\in(0,1)\cap\mathbb{Q}$ is arbitrary, we conclude that $\displaystyle{\limsup_{t\to \infty}t^{-1}u_\theta(t,0,\omega) \le \ol\lambda_j} = \ol H(\theta)$. A symmetric argument provides us with a matching lower bound and completes the proof of \eqref{eq:nebiye}.

\smallskip

In fact, in Lemmas \ref{lem:isbu} and \ref{lem:LBUB}, we resort to a general argument based on Egorov's theorem and ergodicity to deduce that, for every $\theta\in\R$, there exists an $\Omega_0^\theta\in\mathcal{F}$ with $\mathbb{P}(\Omega_0^\theta) = 1$\ such that, for every $\omega\in\Omega_0^\theta$, as $\epsilon>0$, $u_\theta^\epsilon(\,\cdot\,,\,\cdot\,,\omega)$ converges locally uniformly on $[0,\infty)\times\R$ to $\ol u_\theta$ which is defined in \eqref{eq:verya}.
Finally, in Section \ref{sub:homproofs}, under the additional assumption \eqref{eq:duzlip}, we use comparison arguments involving $u_\theta(t,x,\omega)$ and $u_{\theta'}(t,x,\omega)$ with $\theta\in\R$, $\theta'\in\mathbb{Q}$ and $|\theta - \theta'|$ sufficiently small to show that this locally uniform convergence holds on the event $\displaystyle{\Omega_0 = \bigcap_{\theta'\in\mathbb{Q}}\Omega_0^{\theta'}}$ for every $\theta\in\R$,
and thereby complete the proof of Theorem \ref{thm:homlin}. (see Remark \ref{rem:orf}).

\section{Solutions of the auxiliary ODE}\label{sec:ode}

In this section, under the standing assumptions \eqref{A1}, \eqref{A2} and \eqref{H1}--\eqref{H4}, we consider the first-order ODE
\begin{equation}\label{eq:ram}
	a(x,\omega)f' + H(f,x,\omega) = \lambda,\quad x\in \R,
\end{equation}
where $\omega\in\Omega$ and $\lambda\in\R$, and obtain various results regarding its solutions that we will use in the subsequent sections.

\subsection{Solutions in a fixed environment}

\begin{definition}\label{def:ine}
	For every $\omega\in\Omega$, $\lambda\in\R$ and bounded $f_1,f_2:\R\to\R$ such that $f_1 < f_2$ on $\R$, let $\mathcal{S}^\lambda(f_1,f_2,\omega)$ be the set of $f\in\Con^1(\R)$ that solve \eqref{eq:ram} and satisfy $f_1 \le f \le f_2$ on $\R$.
\end{definition}

The following result (covered by \cite[Lemma A.8]{DKY23+}) provides a condition that guarantees that $\mathcal{S}^\lambda(f_1,f_2,\omega)$ is nonempty. 

\begin{lemma}\label{lem:bosdegil}
	For every $\omega\in\Omega$, $\lambda\in\R$ and bounded $f_1,f_2\in\Con^1(\R)$ such that $f_1 < f_2$ on $\R$, assume that either
	\begin{align*}
		&a(x,\omega)f_1' + H(f_1,x,\omega) < \lambda < a(x,\omega)f_2' + H(f_2,x,\omega),\quad x\in \R,\\
		\shortintertext{or}
		&a(x,\omega)f_1' + H(f_1,x,\omega) > \lambda > a(x,\omega)f_2' + H(f_2,x,\omega),\quad x\in \R.
	\end{align*}
	Then, $\mathcal{S}^\lambda(f_1,f_2,\omega)$ is nonempty. 
\end{lemma}

\begin{lemma}\label{lem:eksik}
	 For every $\omega\in\Omega$, $\lambda\in\R$ and bounded $f_1,f_2:\R\to\R$ such that $f_1 < f_2$ on $\R$, the set $\mathcal{S}^\lambda(f_1,f_2,\omega)$ is compact under the topology of locally uniform convergence. Moreover, if $\mathcal{S}^\lambda(f_1,f_2,\omega)$ is nonempty, then the function $\ol f:\R\to\R$, defined by setting
	\[ \ol f(x) = \ol f(x,\omega) = \sup\{f(x):\, f\in\mathcal{S}^\lambda(f_1,f_2,\omega)\}, \]
	is in $\mathcal{S}^\lambda(f_1,f_2,\omega)$.
\end{lemma}

\begin{proof}
	If $\mathcal{S}^\lambda(f_1,f_2,\omega) = \emptyset$, then it is trivially compact. Otherwise, note that $\mathcal{S}^\lambda(f_1,f_2,\omega)$ is equi-bounded by definition. Moreover, equation \eqref{eq:ram} guarantees that, for every $M>0$, $\mathcal{S}^\lambda(f_1,f_2,\omega)$ is equi-Lipschitz on $[-M,M]$.
	Let $(f_n)_{n\ge1}$ be a sequence in $\mathcal{S}^\lambda(f_1,f_2,\omega)$. 
	By the Arzel\'a-Ascoli theorem (and a standard diagonalization argument), it has a subsequence (not relabeled) that converges locally uniformly to some $f\in\Con(\R)$.
	We see from \eqref{eq:ram} that $(f_n')_{n\ge1}$ converges locally uniformly to some $g\in\Con(\R)$. Therefore,
	\[ f(x) - f(0) = \lim_{n\to\infty} \left( f_n(x) - f_n(0) \right) = \lim_{n\to\infty} \int_0^x f_n'(y)dy = \int_0^x g(y)dy \]
	for all $x\in\R$. We conclude that $f\in\Con^1(\R)$, $f' = g$,
	and $f\in\mathcal{S}^\lambda(f_1,f_2,\omega)$. See \cite[Lemma A.9]{DKY23+} for the last assertion.
\end{proof}

For every $\lambda\ge G_L(0)$, let
\begin{equation}\label{eq:ar}
	\ol R^\lambda = \max\{ p\ge0:\,G_L(p) \le \lambda\}.
\end{equation}

\begin{lemma}\label{lem:val}
	For every $\omega\in\Omega$ and $\lambda\in\R$, if $f\in\Con^1(\R)$ is bounded and
	\[ a(x,\omega)f'(x) + H(f(x),x,\omega) \le \lambda\ \text{for all $x\in \R$}, \]
	then $|f(x)| \le \ol R^\lambda$ for all $x\in\R$.
\end{lemma}

\begin{proof}
	Fix any $\epsilon>0$. Observe that, if $|f(x)| \ge \ol R^{\lambda+\epsilon}$ were to hold for any $x\in\R$, then
	\[ a(x,\omega)f'(x) \le \lambda - G_L(f(x)) \le -\epsilon, \]
	and, hence, $f'(x) \le -\epsilon$. It follows that $|f(x)| \le \ol R^{\lambda + \epsilon}$ for all $x\in\R$.
	Indeed:
	\begin{itemize}
		\item [(i)] if there exists an $x_0\in\R$ such that $f(x_0) < - \ol R^{\lambda+\epsilon}$, then $f(x)\to - \infty$ as $x\to\infty$, which contradicts the boundedness of $f$;
		\item [(ii)] if there exists an $x_0\in\R$ such that $f(x_0) > \ol R^{\lambda+\epsilon}$, then $f(x)\to \infty$ as $x\to-\infty$, which contradicts the boundedness of $f$.
	\end{itemize}
	Since $\epsilon>0$ is arbitrary, we are done.
\end{proof}

\subsection{Stationary solutions} 

\begin{definition}\label{def:alarca}
	A measurable function $f:\R\times\Omega\to\R$ is said to be a stationary solution of \eqref{eq:ram} if it is stationary, i.e.,
	\[ f(x,\omega) = f(0,\tau_x\omega)\ \text{for all}\ \omega\in\Omega\ \text{and}\ x\in\R, \]
	$f(\,\cdot\,,\omega)\in\Con^1(\R)$ for all $\omega\in\Omega$,
	and $f(\,\cdot\,,\omega)$ is a bounded solution of \eqref{eq:ram} for $\mathbb{P}$-a.e.\ $\omega$.
	Stationary lower and upper solutions of \eqref{eq:ram} are similarly defined with the equality sign in \eqref{eq:ram} replaced with ``$<$" and ``$>$", respectively.
	The set of stationary solutions (resp.\ lower solutions, upper solutions) of \eqref{eq:ram} is denoted by $\mathcal{S}_{st}^\lambda$ (resp.\ $\mathcal{S}_{st,-}^\lambda$, $\mathcal{S}_{st,+}^\lambda$).
\end{definition}

\begin{remark}\label{rem:grua} 
	For every stationary $f:\R\times\Omega\to\R$ such that $f(\,\cdot\,,\omega)\in\Con(\R)$ for all $\omega\in\Omega$,
	\begin{equation}\label{eq:erg}
		\lim_{x\to\pm\infty}x^{-1}\int_0^xf(y,\omega)dy = \mathbb{E}[f(0,\omega)]
	\end{equation}
	for $\mathbb{P}$-a.e.\ $\omega$ by Birkhoff's ergodic theorem. Hence,
	for every $f_0\in\mathcal{S}_{st}^\lambda$, $f_1\in\mathcal{S}_{st,-}^\lambda$, $f_2\in\mathcal{S}_{st,+}^\lambda$, 
	\begin{align*}
		\Omega_{f_0}^\lambda &= \{ \omega\in\Omega:\,f_0(\,\cdot\,,\omega)\ \text{is a bounded solution of \eqref{eq:ram} that satisfies \eqref{eq:erg}} \},\\
		\Omega_{f_1,-}^\lambda &= \{ \omega\in\Omega:\,f_1(\,\cdot\,,\omega)\ \text{is a bounded lower solution of \eqref{eq:ram} that satisfies \eqref{eq:erg}} \},\\
		\Omega_{f_2,+}^\lambda &= \{ \omega\in\Omega:\,f_2(\,\cdot\,,\omega)\ \text{is a bounded upper solution of \eqref{eq:ram} that satisfies \eqref{eq:erg}} \}
	\end{align*}
	have probability $1$. These events are invariant under $(\tau_z)_{z\in\R}$ by \eqref{A1} and \eqref{H1}.
\end{remark}

\begin{theorem}\label{thm:zek}
	For every $\lambda\in\R$ and $f_1,f_2:\R\times\Omega\to\R$, assume that the event
	\begin{equation}\label{eq:ahanda}
		\tilde\Omega_{1,2} = \{\omega\in\Omega:\ f_1(x,\omega) < f_2(x,\omega)\ \text{for all}\ x\in\R\}
	\end{equation}
	has probability $1$, and either
	\begin{align}
		&\text{$f_1\in\mathcal{S}_{st,-}^\lambda$ and $f_2\in\mathcal{S}_{st,+}^\lambda$,}\label{eq:hu1}\\
		\shortintertext{or}
		&\text{$f_1\in\mathcal{S}_{st,+}^\lambda$ and $f_2\in\mathcal{S}_{st,-}^\lambda$.}\label{eq:hu2}
	\end{align}
	Let $\Omega_{1,2} = \tilde\Omega_{1,2} \cap\Omega_{f_1,-}^\lambda\cap\Omega_{f_2,+}^\lambda$ (resp.\ $\Omega_{1,2} = \tilde\Omega_{1,2} \cap\Omega_{f_1,+}^\lambda\cap\Omega_{f_2,-}^\lambda$) under \eqref{eq:hu1} (resp.\ \eqref{eq:hu2}).
	Then, there exists an $\ol f\in\mathcal{S}_{st}^\lambda$ such that, for every $\omega\in\Omega_{1,2}$,  $\ol f(\,\cdot\,,\omega)\in\mathcal{S}^\lambda(f_1(\,\cdot\,,\omega),f_2(\,\cdot\,,\omega),\omega)$, and it is given by
	\begin{equation}\label{eq:kau2}
		\ol f(x,\omega) = \sup\{f(x):\, f\in\mathcal{S}^\lambda(f_1(\,\cdot\,,\omega),f_2(\,\cdot\,,\omega),\omega)\}.
	\end{equation}
\end{theorem}

\begin{proof}
	For every $\omega\in\Omega_{1,2}$, the set $\mathcal{S}^\lambda(f_1(\,\cdot\,,\omega),f_2(\,\cdot\,,\omega),\omega)$ is nonempty by Lemma \ref{lem:bosdegil}. Hence, the function $\ol f(\,\cdot\,,\omega)$, defined by \eqref{eq:kau2}, is in $\mathcal{S}^\lambda(f_1(\,\cdot\,,\omega),f_2(\,\cdot\,,\omega),\omega)$ by Lemma \ref{lem:eksik}.
	
	Note that the function $\ol f:\R\times\Omega_{1,2}\to\R$ is measurable. This is justified exactly as in the proof of \cite[Lemma 4.4]{DKY23+}.
	
	For every $\omega\in\Omega_{1,2}$ and $y\in\R$,
	\begin{equation}\label{eq:adria}
		f\in\mathcal{S}^\lambda(f_1(\,\cdot\,,\omega),f_2(\,\cdot\,,\omega),\omega)\quad\iff\quad f(\,\cdot + y)\in\mathcal{S}^\lambda(f_1(\,\cdot\,,\tau_y\omega),f_2(\,\cdot\,,\tau_y\omega),\tau_y\omega)
	\end{equation}
	by Definition \ref{def:alarca} and assumptions (A1) and (H1). Therefore,
	\[ \ol f(y,\omega) = \sup\{f(y):\, f(\,\cdot + y)\in\mathcal{S}^\lambda(f_1(\,\cdot\,,\tau_y\omega),f_2(\,\cdot\,,\tau_y\omega),\tau_y\omega)\} = \ol f(0,\tau_y\omega). \]
	
	Finally, we extend $\ol f$ from $\R\times\Omega_{1,2}$ to $\R\times\Omega$ by setting $\ol f(x,\omega) = 0$ for all $\omega\in\Omega\setminus\Omega_{1,2}$ and $x\in\R$. Since	$\Omega_{1,2}$ has probability $1$ and it is invariant under $(\tau_z)_{z\in\R}$ (see Remark \ref{rem:grua}), this extension (still denoted by $\ol f$) is stationary, and hence, it is in $\mathcal{S}_{st}^\lambda$.
\end{proof}

\begin{corollary}\label{cor:circ}
	For every $\lambda,p_1,p_2\in\R$ such that $p_1 < p_2$, if
	\[ G_U(p_1) < \lambda < G_L(p_2)\quad\text{or}\quad G_U(p_2) < \lambda < G_L(p_1), \]
	then there exists an $f\in\mathcal{S}_{st}^\lambda$ such that $p_1 \le f(x,\omega) \le p_2$ for all $\omega\in\Omega$ and $x\in\R$.
\end{corollary}

\begin{proof}
	Let $f_1(x,\omega) = p_1$ and $f_2(x,\omega) = p_2$ for all $\omega\in\Omega$ and $x\in\R$. We have two cases:
	\begin{itemize}
		\item [(i)] if $G_U(p_1) < \lambda < G_L(p_2)$, then
		\begin{align*}
			&a(x,\omega)f_1' + H(f_1,x,\omega) = H(p_1,x,\omega) \le G_U(p_1) < \lambda\quad\text{and}\\
			&a(x,\omega)f_2' + H(f_2,x,\omega) = H(p_2,x,\omega) \ge G_L(p_2) > \lambda,
		\end{align*}
		i.e., \eqref{eq:hu1} holds with $\Omega_{f_1,-}^\lambda = \Omega_{f_2,+}^\lambda = \Omega$;
		\item [(ii)] if $G_U(p_2) < \lambda < G_L(p_1)$, then
		\begin{align*}
			&a(x,\omega)f_1' + H(f_1,x,\omega) = H(p_1,x,\omega) \ge G_L(p_1) > \lambda\quad\text{and}\\
			&a(x,\omega)f_2' + H(f_2,x,\omega) = H(p_2,x,\omega) \le G_U(p_2) < \lambda,
		\end{align*}
		i.e., \eqref{eq:hu2} holds with $\Omega_{f_1,+}^\lambda = \Omega_{f_2,-}^\lambda = \Omega$.
	\end{itemize}
	Hence, the assertion follows from Theorem \ref{thm:zek}.
\end{proof}

\section{The effective Hamiltonian and the bridging lemma}\label{sec:effective}

In this section, under the standing assumptions \eqref{A1}, \eqref{A2} and \eqref{H1}--\eqref{H4}, we construct our candidate for the effective Hamiltonian, obtain some properties of it, and provide a bridging lemma which is a crucial part of our proof of homogenization (which we will complete in Section \ref{sec:homogenization}). 

\subsection{The effective Hamiltonian} 

Recall Definition \ref{def:alarca}. Let
\[ \mathcal{S}_{st} = \bigcup_{\lambda\in\R}\mathcal{S}_{st}^\lambda\quad\text{and}\quad E = \{ \mathbb{E}[f(0,\omega)]:\, f\in\mathcal{S}_{st}\}. \]

\begin{lemma}\label{lem:support}
	$E$ is nonempty, $\inf E = -\infty$ and $\sup E = \infty$.
\end{lemma}

\begin{proof}
	By the coercivity of $G_L$, for every $p\in\R$, there exist $p_1,p_2\in\R$ such that $p_1 < p < p_2$ and $G_U(p) < \min\{ G_L(p_1),G_L(p_2) \}$. Choose a $\lambda\in\R$ such that $G_U(p) < \lambda < \min\{ G_L(p_1),G_L(p_2) \}$. By Corollary \ref{cor:circ}, there exist $f_1,f_2\in\mathcal{S}_{st}^\lambda$ such that $p_1 \le f_1(x,\omega) \le p \le f_2(x,\omega) \le p_2$ for all $\omega\in\Omega$ and $x\in\R$. Note that $\mathbb{E}[f_1(0,\omega)] \le p \le \mathbb{E}[f_2(0,\omega)]$. Therefore, $\inf E \le p \le \sup E$. Since $p\in\R$ is arbitrary, we are done.
\end{proof}

\begin{lemma}\label{lem:icab}
	 For every $\theta\in E$, there is a unique $\lambda(\theta)\in\R$ and a unique $f_\theta\in\mathcal{S}_{st}^{\lambda(\theta)}$ such that $\mathbb{E}[f_\theta(0,\omega)] = \theta$ and 
	\begin{equation}\label{eq:hemcins}
		a(x,\omega)f_\theta'(x,\omega) + H(f_\theta(x,\omega),x,\omega) = \lambda(\theta),\quad x\in\R,
	\end{equation}
	for all $\omega\in\Omega_{f_\theta}^{\lambda(\theta)}$.
\end{lemma}

\begin{proof}
	For every $\lambda_1,\lambda_2\in\R$ (not necessarily distinct), if $f_1\in\mathcal{S}_{st}^{\lambda_1}$ and $f_2\in\mathcal{S}_{st}^{\lambda_2}$, then exactly one of the following mutually exclusive events has probability 1 (see \cite[Lemma A.4]{DKY23+}):
	\begin{align*}
		\{\omega\in\Omega_{f_1}^{\lambda_1}\cap\Omega_{f_2}^{\lambda_2}:\,(f_1 - f_2)(x,\omega) = 0\ \text{for all}\ x\in\R\};\\
		\{\omega\in\Omega_{f_1}^{\lambda_1}\cap\Omega_{f_2}^{\lambda_2}:\,(f_1 - f_2)(x,\omega) < 0\ \text{for all}\ x\in\R\};\\
    	\{\omega\in\Omega_{f_1}^{\lambda_1}\cap\Omega_{f_2}^{\lambda_2}:\,(f_1 - f_2)(x,\omega) > 0\ \text{for all}\ x\in\R\}.
	\end{align*}
	Hence, for every $\theta\in E$, there is a unique $f_\theta\in\mathcal{S}_{st}$ such that $\mathbb{E}[f_\theta(0,\omega)] = \theta$, and the assertion follows.
\end{proof}

\begin{lemma}\label{lem:emi}
	For every $\theta\in E$,
	\[ \mathbb{P}(G_L( f_\theta(x,\omega)) \le \lambda(\theta) \le G_U( f_\theta(x,\omega))\ \text{for all $x\in\R$}) = 1. \]
	Consequently, $G_L(\theta) \le \lambda(\theta) \le G_U( \theta)$.
\end{lemma}

\begin{proof}
	For every $\omega\in\Omega_{f_\theta}^{\lambda(\theta)}$ and $x\in\R$,
	\begin{equation}\label{eq:sevval}
		|f_\theta(x,\omega)| \le \ol R^{\lambda(\theta)}
	\end{equation}
	by Lemma \ref{lem:val}, i.e., $G_L( f_\theta(x,\omega)) \le \lambda(\theta)$.
	
	For every $\epsilon>0$, let
	\begin{align*}
		\ul R^{\lambda(\theta) - \epsilon} &= \min\{ p\ge0:\,G_U(p) \ge \lambda(\theta) - \epsilon\}\\
		\shortintertext{and}
		A^{\lambda(\theta) - \epsilon}(\omega) &= \{ x\in\R:\,|f_\theta(x,\omega)| \le \ul R^{\lambda(\theta) - \epsilon} \}.
	\end{align*}
	By ergodicity, $\mathbb{P}(A^{\lambda(\theta) - \epsilon}(\omega) \neq\emptyset) \in\{0,1\}$. Suppose (for the sake of reaching a contradiction) that $\mathbb{P}(A^{\lambda(\theta) - \epsilon}(\omega) \neq\emptyset) = 1$. Then, $\sup A^{\lambda(\theta) - \epsilon}(\omega) = \infty$ for $\mathbb{P}$-a.e.\ $\omega$. Fix such an $\omega\in\Omega_{f_\theta}^{\lambda(\theta)}$. Note that, for every $x\in A^{\lambda(\theta) - \epsilon}(\omega)$,
	\[ a(x,\omega)f_\theta'(x,\omega) \ge \lambda(\theta) - G_U(f_\theta(x,\omega)) \ge \epsilon, \]
	and, hence, $f_\theta'(x,\omega) \ge \epsilon$. 
	
	Fix any $x_0\in A^{\lambda(\theta) - \epsilon}(\omega)$. Let $y = \inf\{ x\ge x_0:\,x\notin A^{\lambda(\theta) - \epsilon}(\omega) \}$. Then, $y\le x_0 + 2\epsilon^{-1}\ul R^{\lambda(\theta) - \epsilon} + 1$. Indeed, otherwise,
	\begin{align*}
		f_\theta(x_0 + 2\epsilon^{-1}\ul R^{\lambda(\theta) - \epsilon} + 1,\omega) &= f_\theta(x_0,\omega) + \int_{x_0}^{x_0 + 2\epsilon^{-1}\ul R^{\lambda(\theta) - \epsilon} + 1} f_\theta'(x,\omega)dx\\
		&\ge -\ul R^{\lambda(\theta) - \epsilon} + \epsilon(2\epsilon^{-1}\ul R^{\lambda(\theta) - \epsilon} + 1) = \ul R^{\lambda(\theta) - \epsilon} + \epsilon,
	\end{align*}
	which is a contradiction.
	
	Note that $f_\theta(y,\omega) = \ul R^{\lambda(\theta) - \epsilon}$ and $f_\theta'(y,\omega) \ge \epsilon$. Let $z = \inf\{ x > y:\, f_\theta(x,\omega) \le \ul R^{\lambda(\theta) - \epsilon}\} > y$. We claim that $z = \infty$. Indeed, if $z < \infty$, then $f_\theta(z,\omega) = \ul R^{\lambda(\theta) - \epsilon}$ and $f_\theta'(z,\omega) \ge \epsilon$, so there exists a $\delta\in(0,z-y)$ such that $f_\theta(x,\omega) < \ol R^{\lambda(\theta) - \epsilon}$ for all $x\in(z-\delta,z)$, contradicting the definition of $z$. Therefore, $f_\theta(x,\omega) > \ul R^{\lambda(\theta) - \epsilon}$ for all $x > y$, i.e., $\sup A^{\lambda(\theta) - \epsilon}(\omega) = y < \infty$, contradicting our choice of $\omega$. We deduce that $\mathbb{P}(A^{\lambda(\theta) - \epsilon}(\omega) =\emptyset) = 1$, i.e.,
	\[ \mathbb{P}(G_U(f_\theta(x,\omega)) \ge \lambda(\theta) - \epsilon\ \text{for all $x\in\R$}) = 1. \]
	Since $\epsilon > 0$ is arbitrary, we conclude that $\mathbb{P}(G_U(f_\theta(x,\omega)) \ge \lambda(\theta)\ \text{for all $x\in\R$}) = 1$. In combination with the first paragraph, this proves the first assertion.
	
	The second assertion follows from the first one. This is trivial if $\mathbb{P}(f_\theta(x,\omega) = \theta\ \text{for all}\ x\in\R) = 1$. Otherwise,
	\[ \theta = \mathbb{E}[f_\theta(0,\omega)] = \lim_{x\to\infty}x^{-1}\int_0^xf_\theta(y,\omega)dy \in (\inf_{x\in\R} f_\theta(x,\omega),\sup_{x\in\R}f_\theta(x,\omega)) \]
	for $\mathbb{P}$-a.e.\ $\omega$, so $f_\theta(x,\omega) = \theta$ for some $x\in\R$ by the intermediate value theorem.
\end{proof}

\begin{lemma}\label{lem:closed}
	$E$ is a closed subset of $\R$. 
\end{lemma}

\begin{proof}
	Let $(\theta_n)_{n\ge1}$ be a sequence in $E$ that converges monotonically to some $\theta\in\R$. Assume for definiteness that $\theta_n\to\theta-$.
	For the sake of brevity, write $\lambda_n = \lambda(\theta_n)$, $f_n = f_{\theta_n} \in \mathcal{S}_{st}^{\lambda_n}$ and $\tilde \Omega_n = \Omega_{f_n}^{\lambda_n}$ (see Remark \ref{rem:grua}).
	Let $\tilde \Omega_0 = \bigcap_{n\ge1}\tilde\Omega_n$. Note that $\mathbb{P}(\tilde\Omega_0) = 1$ and it is invariant under $(\tau_z)_{z\in\R}$.
	
	By \cite[Lemma A.4]{DKY23+}, the event
	\[ \hat\Omega_0 = \{\omega\in\tilde\Omega_0:\,(f_n - f_{n+1})(x,\omega) < 0\ \text{for all $n\ge1$ and $x\in\R$}\} \]
	has probability $1$, and it is invariant under $(\tau_z)_{z\in\R}$.
	
	By Lemma \ref{lem:emi}, $(\lambda_n)_{n\ge1}$ is a bounded sequence, so it has a subsequence (not relabeled) that converges to some $\lambda\in\R$.
	For every $\omega\in\hat\Omega_0$, $(f_n(\,\cdot\,,\omega))_{n\ge1}$ is equi-bounded by \eqref{eq:sevval}. Moreover, under our choice of notation, \eqref{eq:hemcins} becomes
	\begin{equation}\label{eq:ramen}
		a(x,\omega)f_n'(x,\omega) + H(f_n(x,\omega),x,\omega) = \lambda_n,\quad x\in\R,
	\end{equation}
	from which we deduce that $(f_n(\,\cdot\,,\omega))_{n\ge1}$ is equi-Lipschitz on $[-M,M]$ for all $M>0$.
	
	We define $f:\R\times\hat\Omega_0\to\R$ by
	\begin{equation}\label{eq:mon}
		f(x,\omega) = \lim_{n\to\infty}f_n(x,\omega) = \sup_{n\ge1}f_n(x,\omega).
	\end{equation}
	It follows that $f(\,\cdot\,,\omega)\in\Lip(\R)$. By Dini's theorem, the monotone convergence in \eqref{eq:mon} is locally uniform on $\R$.
	Moreover, we see from \eqref{eq:ramen} that $f_n'(\,\cdot\,,\omega)$ converges locally uniformly to some $g(\,\cdot\,,\omega) \in \Con(\R)$. Observe that
	\[ f(x,\omega) - f(0,\omega) = \lim_{n\to\infty} \left( f_n(x,\omega) - f_n(0,\omega) \right) = \lim_{n\to\infty} \int_0^x f_n'(y,\omega)dy = \int_0^x g(y,\omega)dy \]
	for all $x\in\R$. Therefore, $f(\,\cdot\,,\omega)\in\Con^1(\R)$, $f'(\,\cdot\,,\omega) = g(\,\cdot\,,\omega)$, and
	\[ a(x,\omega)f'(x,\omega) + H(f(x,\omega),x,\omega) = \lambda\ \text{for all}\ x\in\R.\]
	
	For every $\omega\in\hat\Omega_0$ and $x\in\R$,
	\[ f(x,\omega) = \lim_{n\to\infty}f_n(x,\omega) = \lim_{n\to\infty}f_n(0,\tau_x\omega) = f(0,\tau_x\omega). \]
	
	Finally, we extend $f$ from $\R\times\hat\Omega_0$ to $\R\times\Omega$ by setting $f(x,\omega) = 0$ for all $\omega\in\Omega\setminus\hat\Omega_0$ and $x\in\R$. Since $\hat\Omega_0$ has probability $1$ and it is invariant under $(\tau_z)_{z\in\R}$, this extension (still denoted by $f$) is stationary, and hence, it is in $\mathcal{S}_{st}^\lambda$. By the bounded convergence theorem,
	\[ \theta = \lim_{n\to\infty}\theta_n = \lim_{n\to\infty} \mathbb{E}[f_n(0,\omega)] = \mathbb{E}[f(0,\omega)] \in E. \qedhere \]
\end{proof}

We write 
\begin{equation}\label{eq:gs}
	E^c = \bigcup_{j\in J}(\theta_L^{(j)},\theta_R^{(j)}),
\end{equation}
i.e., as a disjoint union of open intervals, where the index set $J\in\mathbb{N}$ is finite (possibly empty) or countably infinite. 

\begin{lemma}\label{lem:esit}
	$\lambda(\theta_L^{(j)}) = \lambda(\theta_R^{(j)})$ for all $j\in J$. 
\end{lemma}

\begin{proof}
	Suppose that $\lambda(\theta_L^{(j)}) \neq \lambda(\theta_R^{(j)})$ for some $j\in J$. For the sake of notational brevity, write $\theta_1 = \theta_L^{(j)}$, $\theta_2=\theta_R^{(j)}$, $\lambda_1 = \lambda(\theta_1)$, $\lambda_2 = \lambda(\theta_2)$, $f_1 = f_{\theta_1}$ and $f_2 = f_{\theta_2}$. Since $\theta_1 < \theta_2$, the event $\tilde\Omega_{1,2}$ (see \eqref{eq:ahanda})
	has probability $1$ by \cite[Lemma A.4]{DKY23+}. Let $\Omega_{1,2} = \tilde\Omega_{1,2} \cap\Omega_{f_1}^{\lambda_1}\cap\Omega_{f_2}^{\lambda_2}$ and $\lambda_0 = \frac{\lambda_1 + \lambda_2}{2}$. There are two cases:
	\begin{itemize}
		\item [(i)] if $\lambda_1 < \lambda_2$, then $f_1\in\mathcal{S}_{st,-}^{\lambda_0}$ and $f_2\in\mathcal{S}_{st,+}^{\lambda_0}$;
		\item [(ii)] if $\lambda_1 > \lambda_2$, then $f_1\in\mathcal{S}_{st,+}^{\lambda_0}$ and $f_2\in\mathcal{S}_{st,-}^{\lambda_0}$.
	\end{itemize}
	By Theorem \ref{thm:zek}, there exists an $f_0\in\mathcal{S}_{st}^{\lambda_0}$ such that $f_1(\,\cdot\,,\omega) \le f_0(\,\cdot\,,\omega) \le f_2(\,\cdot\,,\omega)$ for all $\omega\in\Omega_{1,2}$.  In fact, $f_1(\,\cdot\,,\omega) < f_0(\,\cdot\,,\omega) < f_2(\,\cdot\,,\omega)$ for $\mathbb{P}$-a.e.\ $\omega\in\Omega_{1,2}$ since $\lambda_0\not\in\{ \lambda_1,\lambda_2\}$ (see \cite[Lemma A.4]{DKY23+}). Therefore, $\mathbb{E}[f_0(0,\omega)] \in E\cap(\theta_1,\theta_2) = \emptyset$, which is a contradiction.
\end{proof}

We define $\ol H:\R\to\R$ by setting
\begin{equation}\label{eq:alman}
	\ol H(\theta) = \begin{cases} \lambda(\theta)&\text{if}\ \theta\in E,\\
		\ol\lambda_j &\text{if}\ \theta\in(\theta_L^{(j)},\theta_R^{(j)})\ \text{for any}\ j\in J,
	\end{cases}
\end{equation}
where $\ol\lambda_j = \lambda(\theta_L^{(j)}) = \lambda(\theta_R^{(j)})$ (see \eqref{eq:gs} and Lemma \ref{lem:esit}). We have hereby constructed our candidate for the effective Hamiltonian, which is locally Lipschitz by the following lemma. 

\begin{lemma}\label{lem:lam}
	For every $\theta_1,\theta_2\in E$,
	\begin{equation}\label{eq:emi}
		|\lambda(\theta_1) - \lambda(\theta_2)| \le K_R|\theta_1 - \theta_2|,
	\end{equation}
	where $R = \max\{\ol R^{G_U(\theta_1)},\ol R^{G_U(\theta_2)}\}$ with $\ol{R}^\lambda$ defined in \eqref{eq:ar}. Consequently, $\ol H\in\Liploc(\R)$.
\end{lemma}

\begin{proof}
	Recall from \eqref{eq:sevval} and Lemma \ref{lem:emi} that, for every $\theta\in E$ and $\omega\in\Omega_{f_\theta}^{\lambda(\theta)}$,
	\[ \|f_\theta(\,\cdot\,,\omega)\|_\infty \le \ol R^{\lambda(\theta)} \le \ol R^{G_U(\theta)}. \]
	Fix any $\theta_1,\theta_2\in E$. It suffices to consider $\theta_1 < \theta_2$. There are three cases:
	\begin{itemize}
		\item [(i)] If $\lambda(\theta_1) = \lambda(\theta_2)$, then \eqref{eq:emi} holds trivially.
		\item [(ii)] If $\lambda(\theta_1) < \lambda(\theta_2)$, then fix any $\epsilon \in (0, \lambda(\theta_2)- \lambda(\theta_1))$. Applying \cite[Lemma A.5]{DKY23+} with $f_1 = f_{\theta_1}$ and $f_2 = f_{\theta_2}$, we deduce that
		\begin{equation}\label{eq:salih}
			\mathbb{P}((f_{\theta_1} - f_{\theta_2})(x,\omega) < -K_R^{-1}\epsilon\ \text{for all}\ x\in\R) = 1,
		\end{equation}
		where $R = \max\{\ol R^{G_U(\theta_1)},\ol R^{G_U(\theta_2)}\}$. Therefore,
		\begin{equation}\label{eq:nova}
			\theta_1 - \theta_2 = \mathbb{E}[(f_{\theta_1} - f_{\theta_2})(0,\omega)] \le -K_R^{-1}\epsilon.
		\end{equation}
		Since $\epsilon \in (0, \lambda(\theta_2)- \lambda(\theta_1))$ is arbitrary, we conclude that $\theta_1 - \theta_2 \le -K_R^{-1}(\lambda(\theta_2) - \lambda(\theta_1))$, which implies \eqref{eq:emi}.
		\item [(iii)] If $\lambda(\theta_1) > \lambda(\theta_2)$, then fix any $\epsilon \in (0, \lambda(\theta_1)- \lambda(\theta_2))$. Applying \cite[Lemma A.5]{DKY23+} with $f_1 = f_{\theta_2}$ and $f_2 = f_{\theta_1}$, we deduce that \eqref{eq:salih} holds, and we have \eqref{eq:nova}. Since $\epsilon \in (0, \lambda(\theta_1)- \lambda(\theta_2))$ is arbitrary, we conclude that $\theta_1 - \theta_2 \le -K_R^{-1}(\lambda(\theta_1) - \lambda(\theta_2))$, which implies \eqref{eq:emi}. \qedhere
	\end{itemize}
\end{proof}

\subsection{The bridging lemma}

We will use the following existence result in the proof of Lemma \ref{lem:LBUB} to show that the function $\ol H$ (see \eqref{eq:alman}) indeed gives the effective Hamiltonian on each connected component of $E^c$. 

\begin{lemma}\label{lem:kibbe}
	Suppose that $J\neq\emptyset$ and fix any $j\in J$. For the sake of convenience, write
	$\theta_1 = \theta_L^{(j)}$, $\theta_2=\theta_R^{(j)}$, $f_1=f_{\theta_1}$ and $f_2=f_{\theta_2}$.
	Then, there exists an $\ol\Omega_j\in\mathcal{F}$ with $\mathbb{P}(\ol\Omega_j) = 1$ such that the following hold for every $\omega\in\ol\Omega_j$:
	\begin{itemize}
		\item [(a)] for every $\delta\in(0,1)\cap\mathbb{Q}$, there exist $\ul z_1,\ol z_2\in\R$ with $\ul z_1 < \ol z_2$ and $f_{1,2}^{\ol\lambda_j + \delta}(\,\cdot\,,\omega)\in\Con^1([\ul z_1,\ol z_2])$ such that
		\begin{equation}\label{eq:sahlepvis}
			a(x,\omega)(f_{1,2}^{\ol\lambda_j + \delta})'(x,\omega) + H(f_{1,2}^{\ol\lambda_j + \delta}(x,\omega),x,\omega) = \ol\lambda_j + \delta\ \text{for all $x\in[\ul z_1,\ol z_2]$},
		\end{equation}
		$f_{1,2}^{\ol\lambda_j + \delta}(\ul z_1,\omega) = f_1(\ul z_1,\omega)$ and $f_{1,2}^{\ol\lambda_j + \delta}(\ol z_2,\omega) = f_2(\ol z_2,\omega)$;
		\item [(b)] for every $\delta\in(0,1)\cap\mathbb{Q}$, there exist $\ul z_2,\ol z_1\in\R$ with $\ul z_2 < \ol z_1$ and $f_{2,1}^{\ol\lambda_j - \delta}(\,\cdot\,,\omega)\in\Con^1([\ul z_2,\ol z_1])$ such that
		\begin{equation}\label{eq:kahvevis}
			a(x,\omega)(f_{2,1}^{\ol\lambda_j - \delta})'(x,\omega) + H(f_{2,1}^{\ol\lambda_j - \delta}(x,\omega),x,\omega) = \ol\lambda_j - \delta\ \text{for all $x\in[\ul z_2,\ol z_1]$},
		\end{equation}
		$f_{2,1}^{\ol\lambda_j - \delta}(\ul z_2,\omega) = f_2(\ul z_2,\omega)$ and $f_{2,1}^{\ol\lambda_j - \delta}(\ol z_1,\omega) = f_1(\ol z_1,\omega)$.
	\end{itemize}
\end{lemma}

\begin{proof}
	Note that the event $\Omega_{1,2}^{\ol\lambda_j} = \tilde\Omega_{1,2} \cap\Omega_{f_1}^{\ol\lambda_j}\cap\Omega_{f_2}^{\ol\lambda_j}$ (with $\tilde\Omega_{1,2}$ defined in \eqref{eq:ahanda}) has probability $1$ and it is invariant under $(\tau_z)_{z\in\R}$. With the simplified notation in the statement of the lemma, for every $i\in\{1,2\}$, $\omega\in\Omega_{1,2}^{\ol\lambda_j}$ and $x\in\R$,
	\begin{equation}\label{eq:kamgoz}
		a(x,\omega)f_i'(x,\omega) + H(f_i(x,\omega),x,\omega) = \ol\lambda_j.
	\end{equation}
	
	Fix any $\delta\in(0,1)\cap\mathbb{Q}$. Recall Definition \ref{def:ine}. Let
	\begin{align*}
		\Omega_{1,2}^{\ol\lambda_j + \delta} &= \{\omega\in\Omega_{1,2}^{\ol\lambda_j}:\,\mathcal{S}^{\ol\lambda_j + \delta}(f_1(\,\cdot\,,\omega),f_2(\,\cdot\,,\omega),\omega)\neq\emptyset \}
		\shortintertext{and}
		\Omega_{1,2}^{\ol\lambda_j - \delta} &= \{\omega\in\Omega_{1,2}^{\ol\lambda_j}:\,\mathcal{S}^{\ol\lambda_j - \delta}(f_1(\,\cdot\,,\omega),f_2(\,\cdot\,,\omega),\omega)\neq\emptyset \}.
	\end{align*}
	It follows from \eqref{eq:adria} that $\Omega_{1,2}^{\ol\lambda_j\pm\delta}$ are invariant under $(\tau_z)_{z\in\R}$, and, hence, $\mathbb{P}(\Omega_{1,2}^{\ol\lambda_j\pm\delta}) \in\{0,1\}$ by ergodicity. Suppose (for the sake of reaching a contradiction) that $\Omega_{1,2}^{\ol\lambda_j\pm\delta}$ has probability $1$. Then, by arguing exactly as in the proof of Theorem \ref{thm:zek} (with $\Omega_{1,2}$ replaced with $\Omega_{1,2}^{\ol\lambda_j\pm\delta}$), we deduce the existence of an $f^{\ol\lambda_j\pm\delta}\in\mathcal{S}_{st}^{\ol\lambda_j\pm\delta}$ such that $f_1(\,\cdot\,,\omega) \le f^{\ol\lambda_j\pm\delta}(\,\cdot\,,\omega) \le f_2(\,\cdot\,,\omega)$ for all $\omega\in\Omega_{1,2}^{\ol\lambda_j\pm\delta}$. In fact, $f_1(\,\cdot\,,\omega) < f^{\ol\lambda_j\pm\delta}(\,\cdot\,,\omega) < f_2(\,\cdot\,,\omega)$ for $\mathbb{P}$-a.e.\ $\omega\in\Omega_{1,2}^{\ol\lambda_j\pm\delta}$ (by \cite[Lemma A.4]{DKY23+}) since $\ol\lambda_j\pm\delta \neq \ol\lambda_j$. Therefore, $\mathbb{E}[f^{\ol\lambda_j\pm\delta}(0,\omega)] \in E\cap(\theta_1,\theta_2) = \emptyset$, which is a contradiction. We conclude that
	\begin{equation}\label{eq:doruk}
		\mathbb{P}(\Omega_{1,2}^{\ol\lambda_j\pm\delta}) = 0.
	\end{equation}
	
	\smallskip
	
	We will prove assertion (a). (The proof of assertion (b) is similar.) Fix any $\delta\in(0,1)\cap\mathbb{Q}$.
	For every $\omega\in\Omega_{1,2}^{\ol\lambda_j}$ and $c\in\R$, by the Picard-Lindel\"of theorem, there exists a $c'>c$ such that
	\begin{equation}\label{eq:eksik}
		a(x,\omega)f' + H(f,x,\omega) = \ol\lambda_j + \delta
	\end{equation}
	has a unique $\Con^1$ solution $f(\,\cdot\,,\omega\,|\,c)$ on $[c,c')$ with the initial condition $f(c,\omega\,|\,c) = f_1(c,\omega)$. Note that $f'(c,\omega\,|\,c) > f_1'(c,\omega)$ by comparing \eqref{eq:kamgoz} and \eqref{eq:eksik}. Therefore, $f_1(\,\cdot\,,\omega) < f(\,\cdot\,,\omega\,|\,c) < f_2(\,\cdot\,,\omega)$ on $(c,c')$ if $c' - c > 0$ is sufficiently small.
	Let
	\begin{equation}\label{eq:dmf}
	\begin{aligned} d = d(c,\omega) &= \sup\left\{c' > 0:\, \text{$f(\,\cdot\,,\omega\,|\,c)$ solves \eqref{eq:eksik} on $[c,c']$ and satisfies}\right.\\
													   &\hspace{26mm}\left.\text{$f_1(\,\cdot\,,\omega) < f(\,\cdot\,,\omega\,|\,c) < f_2(\,\cdot\,,\omega)$ on $(c,c')$} \right\}.
	\end{aligned}
	\end{equation}

	Fix any $\omega\in\Omega_{1,2}^{\ol\lambda_j}$ and suppose that $d(c,\omega) = \infty$ for all $c\in\R$. Then, for every $M\ge1$, the sequence $(f(\,\cdot\,,\omega\,|-n))_{n\ge M}$ is equi-bounded on $[-M,M]$. Moreover, it is equi-Lipschitz on $[-M,M]$ by \eqref{eq:eksik}. By the Arzel\'a-Ascoli theorem (and a standard diagonalization argument), $(f(\,\cdot\,,\omega\,|-n))_{n\ge 1}$ has a subsequence (not relabeled) that converges locally uniformly to some $f(\,\cdot\,,\omega)\in\Con(\R)$.
	We see from \eqref{eq:eksik} that $(f'(\,\cdot\,,\omega\,|-n))_{n\ge 1}$ converges locally uniformly to some $g(\,\cdot\,,\omega)\in\Con(\R)$. Therefore,
	\begin{align*}
		f(x,\omega) - f(0,\omega) &= \lim_{n\to\infty} \left( f(x,\omega\,|-n) -  f(0,\omega\,|-n) \right)\\
		&= \lim_{n\to\infty} \int_0^x f'(y,\omega\,|-n)dy = \int_0^x g(y,\omega)dy
	\end{align*}
	for all $x\in\R$. We conclude that $f(\,\cdot\,,\omega)\in\mathcal{S}^{\ol\lambda_j+\delta}(f_1(\,\cdot\,,\omega),f_2(\,\cdot\,,\omega),\omega)$, and, hence, $\omega\in\Omega_{1,2}^{\ol\lambda_j+\delta}$.
	
	The event
	\[ \ol\Omega_j=\!\!\!\! \bigcap_{\delta\in(0,1)\cap\mathbb{Q}}\!\!\!\! \left(\Omega_{1,2}^{\ol\lambda_j}\setminus\Omega_{1,2}^{\ol\lambda_j+\delta}\right) \]
	has probability $1$ by \eqref{eq:doruk}. For every $\omega\in\ol\Omega_j$ and $\delta\in(0,1)\cap\mathbb{Q}$, there exists a $c\in\R$ such that $d = d(c,\omega) < \infty$. By the definition of $d$ in \eqref{eq:dmf}, either $f(d,\omega\,|\,c) = f_1(d,\omega)$ or $f(d,\omega\,|\,c) = f_2(d,\omega)$. We can easily rule out the first scenario. Indeed, if $f(d,\omega\,|\,c) = f_1(d,\omega)$, then $f'(d,\omega\,|\,c) > f_1'(d,\omega)$ by comparing \eqref{eq:kamgoz} and \eqref{eq:eksik}, so $f(x,\omega\,|\,c) < f_1(x,\omega)$ for some $x\in(c,d)$,
	which contradicts the definition of $d$. Hence, $f(d,\omega\,|\,c) = f_2(d,\omega)$. Setting $\ul z_1 = c$, $\ol z_2 = d$ and $f_{1,2}^{\ol\lambda_j + \delta}(\,\cdot\,,\omega) = f(\,\cdot\,,\omega\,|\,c)$ concludes the proof of assertion (a).
\end{proof}

\section{Homogenization}\label{sec:homogenization}

In this section, we work under the standing assumptions \eqref{A1}, \eqref{A2} and \eqref{H1}--\eqref{H4} plus any other assumptions that are explicitly stated in each result. Some of the results do not use all of the standing assumptions (e.g., coercivity of $H$ or stationarity of $a$ and $H$), but for the sake of brevity, we do not keep track of this.

\subsection{Viscosity solutions}\label{sub:viscosity}

We recall some basic definitions regarding viscosity solutions and record a comparison principle. All statements are specialized to our particular setting and purposes. For general background on the theory of viscosity solutions of second-order partial differential equations and its applications, we refer the reader to \cite{users,FleSon06}.

We consider the HJ equation \eqref{eq:birhuzur}, which we rewrite for the sake of convenience:
\begin{equation}\label{eq:genelHJ}
	\partial_tu = a(x,\omega)\partial_{xx}^2u + H(\partial_xu,x,\omega),\quad (t,x)\in(0,\infty)\times\R.
\end{equation}
Here, we take $a:\R\times\Omega\to[0,1]$, i.e., we allow $a$ to vanish, so that \eqref{eq:efhuzur} is covered as well.

\begin{definition}\label{def:vis}
	A function $v\in\Con((0,\infty)\times\R)$ is said to be a viscosity subsolution of \eqref{eq:genelHJ} if, for every $(t_0,x_0)\in(0,\infty)\times\R$ and $\varphi\in\Con^2((0,\infty)\times\R)$ such that $v-\varphi$ attains a local maximum at $(t_0,x_0)$, the following inequality holds:
	\[ \partial_t\varphi(t_0,x_0) \le a(x_0,\omega)\partial_{xx}^2\varphi(t_0,x_0) + H(\partial_x\varphi(t_0,x_0),x_0,\omega). \]
	Similarly, a function $w\in\Con((0,\infty)\times\R)$ is said to be a viscosity supersolution of \eqref{eq:genelHJ} if, for every $(t_0,x_0)\in(0,\infty)\times\R$ and $\varphi\in\Con^2((0,\infty)\times\R)$ such that $w-\varphi$ attains a local minimum at $(t_0,x_0)$, the following inequality holds:
	\[ \partial_t\varphi(t_0,x_0) \ge a(x_0,\omega)\partial_{xx}^2\varphi(t_0,x_0) + H(\partial_x\varphi(t_0,x_0),x_0,\omega). \]
	Finally, a function $u\in\Con((0,\infty)\times\R)$ is said to be a viscosity solution of \eqref{eq:genelHJ} if it is both a viscosity subsolution and a viscosity supersolution of this equation.
\end{definition}

Assumption \eqref{eq:well} (which is used in Corollary \ref{cor:velinim}) involves the following notion.

\begin{definition}\label{def:well}
	We say that the Cauchy problem for \eqref{eq:genelHJ} is well-posed in $\UC([0,\infty)\times\R)$ if the following hold:
	\begin{itemize}
		\item [(i)] for every $g\in\UC(\R)$, \eqref{eq:genelHJ} has a viscosity solution $u\in\UC([0,\infty)\times\R)$ such that $u(0,\,\cdot\,) = g(\,\cdot\,)$ on $\R$;
		\item [(ii)] if $u_1,u_2\in\UC([0,\infty)\times\R)$ are viscosity solutions of \eqref{eq:genelHJ}, then
		\[ \sup\{ |u_1(t,x) - u_2(t,x)|:\,(t,x)\in[0,\infty)\times\R \} = \sup\{ |u_1(0,x) - u_2(0,x)|:\,x\in\R \}. \]
	\end{itemize}
\end{definition}

In the rest of this section, we repeatedly use the following comparison principle. It is covered by, e.g., \cite[Proposition 2.3]{DK17} which is a generalization of \cite[Proposition 1.4]{D19}. 

\begin{proposition}\label{prop:comp}
	Let $v\in\UC([0,\infty)\times\R)$ and $w\in\UC([0,\infty)\times\R)$ be, respectively, a viscosity subsolution and a viscosity supersolution of \eqref{eq:genelHJ}.
	If $\{ v(t,\,\cdot\,):\,t\in[0,\infty) \}$ is an equi-Lipschitz family of functions, i.e.,
	\[ \text{there exists an $\ell>0$ such that}\ |v(t,x) - v(t,y)| \le \ell|x-y|\ \ \text{for all}\ t\in[0,\infty)\ \text{and}\ x,y\in\R, \]
	or $\{ w(t,\,\cdot\,):\,t\in[0,\infty) \}$ is an equi-Lipschitz family of functions, then
	\[ \sup\{ v(t,x) - w(t,x):\,(t,x)\in[0,\infty)\times\R \} = \sup\{ v(0,x) - w(0,x):\,x\in\R \}. \]
\end{proposition}

\subsection{Locally uniform convergence for each $\theta\in E$}\label{sub:firat}

\begin{lemma}\label{lem:equi}
	Assume  \eqref{eq:existunique}. For every $\theta\in\R$, there exists an $\Omega_{\mathrm{ue}}^\theta\in\mathcal{F}$ with $\mathbb{P}(\Omega_{\mathrm{ue}}^\theta) = 1$ such that
	$\{ u_\theta^\epsilon(t,\,\cdot\,,\omega):\,\epsilon\in(0,1],\ t\in[0,\infty),\ \omega\in\Omega_{\mathrm{ue}}^\theta \}$ is a uniformly (in $x$) equicontinuous family of functions.
\end{lemma}

\begin{proof}
	See the proof of \cite[Lemma 4.4]{Y21b} which carries over verbatim.
\end{proof}

\begin{lemma}\label{lem:isbu}
	Assume \eqref{eq:existunique}. For every $\theta\in E$, there exists an $\Omega_0^\theta\in\mathcal{F}$ with $\mathbb{P}(\Omega_0^\theta) = 1$ such that, for every $\omega\in\Omega_0^\theta$ and $T,M>0$,
	\[ \lim_{\epsilon\to0}\sup_{t\in[0,T]}\sup_{x\in[-M,M]} |u_\theta^\epsilon(t,x,\omega) - t\lambda(\theta) - \theta x| = 0. \]
\end{lemma}

\begin{proof}
	Recall Lemma \ref{lem:icab}. Fix any $\theta\in E$. For every $\omega\in\Omega_{f_\theta}^{\lambda(\theta)}$, define $F_\theta(\,\cdot\,,\omega):\R\to\R$ by setting
	\[ F_\theta(x,\omega) = \int_0^xf_\theta(y,\omega)dy. \]
	It follows immediately from \eqref{eq:hemcins} that
	\[ \tilde u_\theta(t,x,\omega) = t\lambda(\theta) + F_\theta(x,\omega) \]
	gives a (classical) solution of \eqref{eq:genelHJ} in $\Lip\cap\Con^2([0,\infty)\times\R)$.
	
	For every $\omega\in\Omega_{f_\theta}^{\lambda(\theta)}$ and $\delta\in(0,1)$, define $\tilde v_{\theta,\delta}(\,\cdot\,,\,\cdot\,,\omega):[0,\infty)\times\R\to\R$ by setting
	\begin{align*}
		\tilde v_{\theta,\delta}(t,x,\omega) &= t(\lambda(\theta)- (K_R + 1)\delta) + F_\theta(x,\omega) - \delta\psi(x) - C\\
		&= \tilde u_\theta(t,x,\omega) - t(K_R + 1)\delta - \delta\psi(x) - C,
	\end{align*}
	where $R = \ol R^{\lambda(\theta)} + 1$ (see \eqref{eq:ar}),
	\[ \psi(x) = \frac{2}{\pi}\int_0^x\arctan(y)dy \]
	which satisfies
	\begin{equation}\label{eq:sofra1}
		0 \le \psi''(\cdot) \le 1,\ \ -1 \le \psi'(\cdot) \le 1,
	\end{equation}
	\begin{equation}\label{eq:sofra2}
		\lim_{x\to-\infty}\psi'(x) = -1\ \ \text{and}\ \ \lim_{x\to\infty}\psi'(x) = 1,
	\end{equation}
	and $C>0$ is to be determined. Note that, for every $(t,x)\in(0,\infty)\times\R$,
	\begin{align}
		&\quad\ a(x,\omega)\partial_{xx}^2\tilde v_{\theta,\delta}(t,x,\omega) + H(\partial_x\tilde v_{\theta,\delta}(t,x,\omega),x,\omega)\nonumber\\
		&= a(x,\omega)(\partial_{xx}^2\tilde u_\theta(t,x,\omega) - \delta\psi''(x)) + H(\partial_x\tilde u_\theta(t,x,\omega) - \delta\psi'(x),x,\omega)\nonumber\\
		&\ge a(x,\omega)\partial_{xx}^2\tilde u_\theta(t,x,\omega) - \delta + H(\partial_x\tilde u_\theta(t,x,\omega),x,\omega) - K_R\delta\label{eq:kullan}\\
		&= \partial_t\tilde u_\theta(t,x,\omega) - (K_R + 1)\delta = \partial_t\tilde v_{\theta,\delta}(t,x,\omega).\nonumber
	\end{align}
	The inequality in \eqref{eq:kullan} follows from \eqref{eq:sofra1} and Lemma \ref{lem:val}. Hence, $\tilde v_{\theta,\delta}(\,\cdot\,,\,\cdot\,,\omega)$ is a subsolution of \eqref{eq:genelHJ} in $\Lip\cap\Con^2([0,\infty)\times\R)$.
	
	For every $\omega\in\Omega_{f_\theta}^{\lambda(\theta)}$ and $\delta\in(0,1)$,
	\[ \lim_{x\to-\infty}x^{-1}\tilde v_{\theta,\delta}(0,x,\omega) = \theta + \delta \quad\text{and}\quad \lim_{x\to\infty}x^{-1}\tilde v_{\theta,\delta}(0,x,\omega) = \theta - \delta \]
	by Remark \ref{rem:grua} and \eqref{eq:sofra2}. Therefore,
	\[ \tilde v_{\theta,\delta}(0,x,\omega) \le \theta x = u_\theta(0,x,\omega) \]
	for all $x\in\R$ when $C = C(\theta,\delta,\omega) > 0$ is sufficiently large. By the comparison principle in Proposition \ref{prop:comp}, 
	\[ \tilde v_{\theta,\delta}(t,x,\omega) \le u_\theta(t,x,\omega)\ \text{for all $(t,x)\in[0,\infty)\times\R$}. \]
	In particular,
	\[ \liminf_{\epsilon\to0}u_\theta^\epsilon(1,0,\omega) = \liminf_{\epsilon\to0}\epsilon u_\theta\left(\frac1{\epsilon},0,\omega\right) \ge \lim_{\epsilon\to0}\epsilon \tilde v_{\theta,\delta}\left(\frac1{\epsilon},0,\omega\right) = \lambda(\theta) - (K_R + 1)\delta. \]
	
	Similarly, for every $\omega\in\Omega_{f_\theta}^{\lambda(\theta)}$ and $\delta\in(0,1)$,
	\[ \tilde w_{\theta,\delta}(t,x,\omega) = t(\lambda(\theta) + (K_R + 1)\delta) + F_\theta(x,\omega) + \delta\psi(x) + C \]
	defines a supersolution of \eqref{eq:genelHJ} in $\Lip\cap\Con^2([0,\infty)\times\R)$, and
	\[ \tilde w_{\theta,\delta}(0,x,\omega) \ge \theta x = u_\theta(0,x,\omega) \]
	for all $x\in\R$ when $C = C(\theta,\delta,\omega) > 0$ is sufficiently large. By the comparison principle in Proposition \ref{prop:comp},
	\[ \limsup_{\epsilon\to0}u_\theta^\epsilon(1,0,\omega) = \limsup_{\epsilon\to0}\epsilon u_\theta\left(\frac1{\epsilon},0,\omega\right) \le \lim_{\epsilon\to0}\epsilon \tilde w_{\theta,\delta}\left(\frac1{\epsilon},0,\omega\right) = \lambda(\theta) + (K_R + 1)\delta. \]
	Since $\delta\in(0,1)$ is arbitrary, we deduce that
	\begin{equation}\label{eq:nokta}
		\lim_{\epsilon\to0}u_\theta^\epsilon(1,0,\omega) = \lambda(\theta).
	\end{equation}
	
	Finally, the desired locally uniform convergence follows from \eqref{eq:nokta} and Lemma \ref{lem:equi} by a general and now standard argument involving Egorov's theorem and the Birkhoff ergodic theorem. See \cite[pp.\ 1501--1502]{KRV} or \cite[Lemma 4.1]{DK17} which is based on \cite[Lemma 2.4]{AT14}.
\end{proof}

\subsection{Locally uniform convergence for each $\theta\in E^c$}\label{sub:sirat}

Recall \eqref{eq:gs}. 

\begin{lemma}\label{lem:LBUB}
	Assume  \eqref{eq:existunique}. Suppose that $J\neq\emptyset$ and fix any $j\in J$. Then, there exists an $\ol\Omega_j\in\mathcal{F}$ with $\mathbb{P}(\ol\Omega_j) = 1$ such that
	\begin{equation}\label{eq:fit}
		\lim_{\epsilon\to0} u_\theta^\epsilon(1,0,\omega) = \ol\lambda_j
	\end{equation}
	for all $\theta\in(\theta_L^{(j)},\theta_R^{(j)})$ and $\omega\in\ol\Omega_j$. Moreover, for every $\theta\in(\theta_L^{(j)},\theta_R^{(j)})$, there exists an $\Omega_0^\theta\in\mathcal{F}$ with $\mathbb{P}(\Omega_0^\theta) = 1$ such that, for every $\omega\in\Omega_0^\theta$ and $T,M>0$,
	\begin{equation}\label{eq:locunif3}
		\lim_{\epsilon\to0}\sup_{t\in[0,T]}\sup_{x\in[-M,M]} |u_\theta^\epsilon(t,x,\omega) - t\ol\lambda_j - \theta x| = 0.
	\end{equation}
\end{lemma}

\begin{proof}
	
	For the sake of convenience, write $\theta_1= \theta_L^{(j)}$, $\theta_2=\theta_R^{(j)}$, $f_1=f_{\theta_1}$ and $f_2=f_{\theta_2}$. 
	We divide the proof into several steps.
	
	\subsubsection*{Asymptotic upper bound at $(1,0)$}

	Recall Lemma \ref{lem:kibbe}(a). For every $\omega\in\ol\Omega_j$ and $\delta\in(0,1)\cap\mathbb{Q}$, define $F_{1,2}^{\ol\lambda_j + \delta}(\,\cdot\,,\omega):\R\to\R$ by setting
	\begin{equation}\label{eq:ferdi}
		F_{1,2}^{\ol\lambda_j + \delta}(x,\omega) = \begin{cases}
										  	\int_{\ul z_1}^xf_1(y,\omega)\,dy &\text{if}\ x \le \ul z_1,\\ 
											\int_{\ul z_1}^xf_{1,2}^{\ol\lambda_j + \delta}(y,\omega)dy &\text{if}\ \ul z_1 < x < \ol z_2,\\
											\int_{\ul z_1}^{\ol z_2}f_{1,2}^{\ol\lambda_j + \delta}(y,\omega)dy +\int_{\ol z_2}^xf_2(y,\omega)\,dy &\text{if}\ x \ge \ol z_2.
										  \end{cases}
	\end{equation}
	Since
	\begin{align*}
		(F_{1,2}^{\ol\lambda_j + \delta})'(\ul z_1-,\omega) &= f_1(\ul z_1,\omega) = f_{1,2}^{\ol\lambda_j + \delta}(\ul z_1,\omega) = (F_{1,2}^{\ol\lambda_j + \delta})'(\ul z_1+,\omega)\\
		\shortintertext{and}
		(F_{1,2}^{\ol\lambda_j + \delta})'(\ol z_2-,\omega) &= f_{1,2}^{\ol\lambda_j + \delta}(\ol z_2,\omega) = f_2(\ol z_2,\omega) = (F_{1,2}^{\ol\lambda_j + \delta})'(\ol z_2+,\omega),
	\end{align*}
	we deduce that $F_{1,2}^{\ol\lambda_j + \delta}(\,\cdot\,,\omega)\in\Lip\cap\Con^1(\R)$. It is easy to check that
	\begin{equation}\label{eq:lem}
		a(x,\omega)(F_{1,2}^{\ol\lambda_j + \delta})''(\,\cdot\,,\omega) + H((F_{1,2}^{\ol\lambda_j + \delta})'(\,\cdot\,,\omega),x,\omega) \le \ol\lambda_j + \delta,\quad x\in \R,
	\end{equation}
	in the viscosity sense. Indeed, at every $x\in\R\setminus\{\ul z_1,\ol z_2\}$, it follows from \eqref{eq:hemcins}, \eqref{eq:sahlepvis} and \eqref{eq:ferdi} that \eqref{eq:lem} holds in the classical sense. For every $\varphi\in\Con^2(\R)$ such that $F_{1,2}^{\ol\lambda_j + \delta}(\,\cdot\,,\omega) - \varphi$ has a local minimum at $\ul z_1$, we see that $\varphi'(\ul z_1) = f_1(\ul z_1,\omega) = f_{1,2}^{\ol\lambda_j + \delta}(\ul z_1,\omega)$, $\varphi''(\ul z_1) \le \min\{f_1'(\ul z_1,\omega),(f_{1,2}^{\ol\lambda_j + \delta})'(\ul z_1,\omega)\}$, and 
	\[ a(\ul z_1,\omega)\phi''(\ul z_1) + H(\phi'(\ul z_1),\ul z_1,\omega) \le a(\ul z_1,\omega)f_1'(\ul z_1,\omega) + H(f_1(\ul z_1,\omega),\ul z_1,\omega) = \ol\lambda_j. \]
	Similarly, for every $\varphi\in\Con^2(\R)$ such that $F_{1,2}^{\ol\lambda_j + \delta}(\,\cdot\,,\omega) - \varphi$ has a local minimum at $\ol z_2$, we see that $\varphi'(\ol z_2) = f_{1,2}^{\ol\lambda_j + \delta}(\ol z_2,\omega) = f_2(\ol z_2,\omega)$, $\varphi''(\ol z_2) \le \min\{(f_{1,2}^{\ol\lambda_j + \delta})'(\ol z_2,\omega),f_2'(\ol z_2,\omega)\}$, and
	\[ a(\ol z_2,\omega)\phi''(\ol z_2) + H(\phi'(\ol z_2),\ol z_2,\omega) \le a(\ol z_2,\omega)f_2'(\ol z_2,\omega) + H(f_2(\ol z_2,\omega),\ol z_2,\omega) = \ol\lambda_j. \]
	It follows immediately from \eqref{eq:lem} that $w_{j,\delta}(\,\cdot\,,\,\cdot\,,\omega):[0,\infty)\times\R\to\R$, defined by
	\[ w_{j,\delta}(t,x,\omega) = t(\ol\lambda_j + \delta) + F_{1,2}^{\ol\lambda_j + \delta}(x,\omega) + C, \]
	where $C > 0$ is to be determined, is a viscosity supersolution of \eqref{eq:genelHJ} in $\Lip\cap\Con^1([0,\infty)\times\R)$.
	
	For every $\omega\in\ol\Omega_j$,
	\[ \lim_{x\to-\infty}x^{-1}w_{j,\delta}(0,x,\omega) = \theta_1 \quad\text{and}\quad \lim_{x\to\infty}x^{-1}w_{j,\delta}(0,x,\omega) = \theta_2 \]
	by \eqref{eq:ferdi} and Remark \ref{rem:grua}. Therefore, for every $\theta\in(\theta_1,\theta_2)$ and $\omega\in\ol\Omega_j$,
	\[ u_\theta(0,x,\omega) = \theta x  \le w_{j,\delta}(0,x,\omega) \]
	for all $x\in\R$ when $C = C(\theta,\delta,\omega) > 0$ is sufficiently large. 
	By the comparison principle in Proposition \ref{prop:comp}, 
	\[ u_\theta(t,x,\omega) \le w_{j,\delta}(t,x,\omega)\ \text{for all $(t,x)\in[0,\infty)\times\R$}. \]
	In particular,
	\begin{equation}\label{eq:duzUB}
		\limsup_{\epsilon\to0}u_\theta^\epsilon(1,0,\omega) = \limsup_{\epsilon\to0}\epsilon u_\theta\left(\frac1{\epsilon},0,\omega\right) \le \lim_{\epsilon\to0}\epsilon w_{j,\delta}\left(\frac1{\epsilon},0,\omega\right) = \ol\lambda_j + \delta.
	\end{equation}
	
	\subsubsection*{Asymptotic lower bound at $(1,0)$}
	
	Recall Lemma \ref{lem:kibbe}(b). For every $\omega\in\ol\Omega_j$ and $\delta\in(0,1)\cap\mathbb{Q}$, define $F_{2,1}^{\ol\lambda_j - \delta}(\,\cdot\,,\omega):\R\to\R$ by setting
	\begin{equation}\label{eq:tayfur}
		F_{2,1}^{\ol\lambda_j - \delta}(x,\omega) = \begin{cases}
			\int_{\ul z_2}^xf_2(y,\omega)\,dy &\text{if}\ x \le \ul z_2,\\ 
			\int_{\ul z_2}^xf_{2,1}^{\ol\lambda_j - \delta}(y,\omega)dy &\text{if}\ \ul z_2 < x < \ol z_1,\\
			\int_{\ul z_2}^{\ol z_1}f_{2,1}^{\ol\lambda_j - \delta}(y,\omega)dy +\int_{\ol z_1}^xf_1(y,\omega)\,dy &\text{if}\ x \ge \ol z_1.
		\end{cases}
	\end{equation}
	It is easy to check (as in the first step of the proof) that $F_{2,1}^{\ol\lambda_j - \delta}(\,\cdot\,,\omega)\in\Lip\cap\Con^1(\R)$ and
	\[ a(x,\omega)(F_{2,1}^{\ol\lambda_j - \delta})''(\,\cdot\,,\omega) + H((F_{2,1}^{\ol\lambda_j - \delta})'(\,\cdot\,,\omega),x,\omega) \ge \ol\lambda_j - \delta,\quad x\in \R, \]
	in the viscosity sense. It follows immediately
	that $v_{j,\delta}(\,\cdot\,,\,\cdot\,,\omega):[0,\infty)\times\R\to\R$, defined by
	\[ v_{j,\delta}(t,x,\omega) = t(\ol\lambda_j - \delta) + F_{2,1}^{\ol\lambda_j - \delta}(x,\omega) - C, \]
	where $C > 0$ is to be determined, is a viscosity subsolution of \eqref{eq:genelHJ} in $\Lip\cap\Con^1([0,\infty)\times\R)$.
	
	For every $\omega\in\ol\Omega_j$,
	\[ \lim_{x\to-\infty}x^{-1}v_{j,\delta}(0,x,\omega) = \theta_2 \quad\text{and}\quad \lim_{x\to\infty}x^{-1}v_{j,\delta}(0,x,\omega) = \theta_1 \]
	by \eqref{eq:tayfur} and Remark \ref{rem:grua}. Therefore, for every $\theta\in(\theta_1,\theta_2)$ and $\omega\in\ol\Omega_j$,
	\[ u_\theta(0,x,\omega) = \theta x  \ge v_{j,\delta}(0,x,\omega) \]
	for all $x\in\R$ when $C = C(\theta,\delta,\omega) > 0$ is sufficiently large. 
	By the comparison principle in Proposition \ref{prop:comp},
	\[ u_\theta(t,x,\omega) \ge v_{j,\delta}(t,x,\omega)\ \text{for all $(t,x)\in[0,\infty)\times\R$}. \]
	In particular,
	\begin{equation}\label{eq:duzLB}
		\liminf_{\epsilon\to0}u_\theta^\epsilon(1,0,\omega) = \liminf_{\epsilon\to0}\epsilon u_\theta\left(\frac1{\epsilon},0,\omega\right) \ge \lim_{\epsilon\to0}\epsilon v_{j,\delta}\left(\frac1{\epsilon},0,\omega\right) = \ol\lambda_j - \delta.
	\end{equation}
	
	\subsubsection*{Convergence at $(1,0)$}
	
	Since $\delta\in(0,1)\cap\mathbb{Q}$ is arbitrary, combining \eqref{eq:duzUB} and \eqref{eq:duzLB}, we conclude that \eqref{eq:fit} holds for all $\theta\in(\theta_1,\theta_2)$ and $\omega\in\ol\Omega_j$.
	
	\subsubsection*{Locally uniform convergence in $(t,x)$} 
	
	Fix any $\theta\in(\theta_1,\theta_2)$. Recall from Lemma \ref{lem:equi} that there exists an $\Omega_{\mathrm{ue}}^\theta\in\mathcal{F}$ with $\mathbb{P}(\Omega_{\mathrm{ue}}^\theta) = 1$ such that $\{ u_\theta^\epsilon(t,\,\cdot\,,\omega):\,\epsilon\in(0,1],\ t\in[0,\infty),\ \omega\in\Omega_{\mathrm{ue}}^\theta \}$ is a uniformly (in $x$) equicontinuous family of functions.
	By the general argument (involving Egorov's theorem and the Birkhoff ergodic theorem) we cited at the end of the proof of Lemma \ref{lem:isbu}, there exists an $\Omega_0^\theta\subset\ol\Omega_j\cap\Omega_{\mathrm{ue}}^\theta$ with $\mathbb{P}\left(\left(\ol\Omega_j\cap\Omega_{\mathrm{ue}}^\theta\right)\setminus\Omega_0^\theta\right) = 0$ (which implies $\mathbb{P}(\Omega_0^\theta) = 1$) such that \eqref{eq:locunif3} holds for all $\omega\in\Omega_0^\theta$ and $T,M>0$.
\end{proof}

\subsection{Completing the proof of homogenization}\label{sub:homproofs}

\begin{proof}[Proof of Theorem \ref{thm:homlin}]
	The function $\ol H$ (defined in \eqref{eq:alman}) is coercive and locally Lipschitz continuous by Lemmas \ref{lem:emi} and \ref{lem:lam}. Therefore, the Cauchy problem for \eqref{eq:efhuzur} is well-posed in $\UC([0,\infty)\times\R)$ (see, e.g., \cite[Theorem 2.5]{DK17}). For every $\theta\in\R$, observe that the unique (classical and hence viscosity) solution $\ol u_\theta$ of \eqref{eq:efhuzur} with the initial condition $\ol u_\theta(x) = \theta x$, $x\in\R$, is given by
	\[ \ol u_\theta(t,x) = t\ol H(\theta) + \theta x. \]
	
	Let
	\[ \Omega_0 = \bigcap_{\theta\in\mathbb{Q}}\Omega_0^\theta \]
	with $\Omega_0^\theta\in\mathcal{F}$ provided in Lemma \ref{lem:isbu} and Lemma \ref{lem:LBUB} when $\theta\in E$ and $\theta\in E^c$, respectively. Note that $\mathbb{P}(\Omega_0) = 1$ and, for every $\omega\in\Omega_0$ and $\theta\in\mathbb{Q}$, as $\epsilon\to0$, $u_\theta^\epsilon(\,\cdot\,,\,\cdot\,,\omega)$ converges locally uniformly on $[0,\infty)\times\R$ to $\ol u_\theta$. It remains to generalize this statement to all $\theta\in\R$.
	
	Fix any $\theta\in\R$. For every $\omega\in\Omega$ and $\delta\in(0,1)$, define $v_{\theta,\delta}(\,\cdot\,,\,\cdot\,,\omega):[0,\infty)\times\R\to\R$ and $w_{\theta,\delta}(\,\cdot\,,\,\cdot\,,\omega):[0,\infty)\times\R\to\R$ by
	\begin{equation}\label{eq:kamp}
		\begin{aligned}
			v_{\theta,\delta}(t,x,\omega) &= u_\theta(t,x,\omega) - t(K_R + 1)\delta - \delta\psi(x) - C\quad\text{and}\\
			w_{\theta,\delta}(t,x,\omega) &= u_\theta(t,x,\omega) + t(K_R + 1)\delta + \delta\psi(x) + C,
		\end{aligned}
	\end{equation}
	where $R = \ell_\theta(\omega) + 1$ (see \eqref{eq:duzlip}),
	\[ \psi(x) = \frac{2}{\pi}\int_0^x\arctan(y)dy \]
	which satisfies \eqref{eq:sofra1}--\eqref{eq:sofra2} from the proof of Lemma \ref{lem:isbu}, and $C>0$ is to be determined. Let us check that $v_{\theta,\delta}(\,\cdot\,,\,\cdot\,,\omega)$ is a viscosity subsolution of \eqref{eq:birhuzur}. For every $(t_0,x_0)\in(0,\infty)\times\R$ and $\varphi\in\Con^2((0,\infty)\times\R)$ such that $v_{\theta,\delta}(\,\cdot\,,\,\cdot\,,\omega)-\varphi$ attains a local maximum at $(t_0,x_0)$, define $\tilde\varphi\in\Con^2((0,\infty)\times\R)$ by
	\[ \tilde\varphi(t,x) = \varphi(t,x) +  t(K_R + 1)\delta + \delta\psi(x) + C \]
	and note that $u_\theta(\,\cdot\,,\,\cdot\,,\omega) - \tilde\varphi = v_{\theta,\delta}(\,\cdot\,,\,\cdot\,,\omega)-\varphi$.
	Therefore, 
	\begin{align*}
		&\quad\ a(x_0,\omega)\partial_{xx}^2\varphi(t_0,x_0) + H(\partial_x\varphi(t_0,x_0),x_0,\omega)\\
		&= a(x_0,\omega)(\partial_{xx}^2\tilde\varphi(t_0,x_0) - \delta\psi''(x)) + H(\partial_x\tilde\varphi(t_0,x_0) - \delta\psi'(x),x_0,\omega)\\
		&\ge a(x_0,\omega)\partial_{xx}^2\tilde\varphi(t_0,x_0) - \delta + H(\partial_x\tilde\varphi(t_0,x_0),x_0,\omega) - K_R\delta\\
		&\ge \partial_t\tilde\varphi(t_0,x_0) - (K_R + 1)\delta = \partial_t\varphi(t_0,x_0).
	\end{align*}
	Similarly, $w_{\theta,\delta}(\,\cdot\,,\,\cdot\,,\omega)$ is a viscosity supersolution of \eqref{eq:birhuzur}. 
	
	Choose any $\theta'\in\mathbb{Q}$ such that $|\theta - \theta'| < \frac{\delta}{2}$. It follows from \eqref{eq:sofra2} that, when $C = C(\delta) > 0$ is sufficiently large, 
	\[ v_{\theta,\delta}(0,x,\omega) = \theta x - \delta\psi(x) - C \le u_{\theta'}(0,x,\omega) = \theta' x \le \theta x + \delta\psi(x) + C = w_{\theta,\delta}(0,x,\omega) \]
	for all $x\in\R$. By the comparison principle in Proposition \ref{prop:comp},
	\[ v_{\theta,\delta}(t,x,\omega) \le u_{\theta'}(t,x,\omega) \le w_{\theta,\delta}(t,x,\omega)\ \text{for all $(t,x)\in[0,\infty)\times\R$}. \]	
	We combine these inequalities with the definitions in \eqref{eq:kamp} and deduce that
	\[ |u_\theta(t,x,\omega) - u_{\theta'}(t,x,\omega)| \le t(K_R + 1)\delta + \delta|x| + C \]
	for all $\omega\in\Omega$ and $(t,x)\in[0,\infty)\times\R$.
	
	Finally, for every $\omega\in\Omega_0$ and $T,M>0$,
	\begin{align*}
		&\quad\, \limsup_{\epsilon\to0}\sup_{t\in[0,T]}\sup_{x\in[-M,M]} |u_\theta^\epsilon(t,x,\omega) - t\ol H(\theta) - \theta x|\\
		&\le \limsup_{\epsilon\to0}\sup_{t\in[0,T]}\sup_{x\in[-M,M]} \left( |u_{\theta'}^\epsilon(t,x,\omega) - t\ol H(\theta') - \theta' x| + |u_\theta^\epsilon(t,x,\omega) - u_{\theta'}^\epsilon(t,x,\omega)| \right)\\
		&\quad + T|\ol H(\theta) - \ol H(\theta')| + |\theta - \theta'|M\\
		&\le T\left[ (K_R + 1)\delta + |\ol H(\theta) - \ol H(\theta')| \right] + \left[ \delta + |\theta - \theta'| \right] M.
	\end{align*}
	Since $|\theta - \theta'|\le\frac{\delta}{2}$, $\ol H\in\Liploc(\R)$ and $\delta\in(0,1)$ is arbitrary, we conclude that
	\begin{equation}\label{eq:yaz}
		\lim_{\epsilon\to0}\sup_{t\in[0,T]}\sup_{x\in[-M,M]} |u_\theta^\epsilon(t,x,\omega) - t\ol H(\theta) - \theta x| = 0.\qedhere
	\end{equation}
\end{proof}

\begin{remark}\label{rem:orf}
	We had already shown in Lemmas \ref{lem:isbu} and \ref{lem:LBUB} that, for every $\theta\in\R$, the locally uniform convergence in \eqref{eq:yaz} holds on a set of probability $1$ that is allowed to depend on $\theta$. In the proof of Theorem \ref{thm:homlin} that we gave above, we show that this locally uniform convergence holds on a set of probability $1$ that is independent of $\theta$. This last step is in fact covered by \cite[Lemma 4.1]{DK17}, albeit under certain additional mild assumptions (which we did not want to impose), most notably that the Lipschitz constant $\ell_\theta$ in \eqref{eq:duzlip} is locally bounded in $\theta$. 
\end{remark}

\begin{proof}[Proof of Corollary \ref{cor:velinim}]
	The desired result follows readily from Theorem \ref{thm:homlin} and \cite[Theorem 3.1]{DK17}. The set $\Omega_0\in\mathcal{F}$ with $\mathbb{P}(\Omega_0) = 1$ is the one in Theorem \ref{thm:homlin}.
\end{proof}

\bibliographystyle{abbrv}
\bibliography{viscousHJ.bib}

\end{document}